\documentclass[12pt, reqno]{amsart}
\usepackage{array, epsfig}
\usepackage{bbm}
\usepackage{hyperref}
\usepackage[numbers,square]{natbib}
\usepackage{color}
\usepackage{mathtools,amssymb,amsthm,mathrsfs,calc,graphicx,xcolor,cleveref,dsfont,tikz,pgfplots,bm}

  \evensidemargin
\oddsidemargin
\newtheorem{theorem}{Theorem}[section]
\newtheorem*{remark}{Remark}
\newtheorem{lemma}[theorem]{Lemma}
\newtheorem{corollary}[theorem]{Corollary}
\newtheorem{proposition}[theorem]{Proposition}
\newcommand{\Sp}{\mathbb S}
\newcommand{\R}{\mathbb{R}}

\newcommand{\C}{\mathbb{C}}

\newcommand{\N}{\mathbb{N}}

\newcommand{\Ex}{\mathbb{E}\,}
\renewcommand{\Pr}{\mathbb{P}\,}

\renewcommand{\Re}{\text{Re}\,}

\newcommand{\dd}{{\rm d}}

\newcommand{\V}{\mathop{\mathrm{V}}\nolimits}
\newcommand{\Var}{\mathop{\mathrm{Var}}\nolimits}
\newcommand{\Gam}{\mathop{\mathrm{Gamma}}\nolimits}
\newcommand{\Bet}{\mathop{\mathrm{Beta}}\nolimits}

\newcommand{\inter}{\mathop{\mathrm{int}}\nolimits}
\newcommand{\cl}{\mathop{\mathrm{cl}}\nolimits}

\newcommand{\Simpl}{\operatorname{Simpl}}

\newcommand{\eqdistr}{\stackrel{D}{=}}

\newcommand{\conv}{\mathop{\mathrm{conv}}\nolimits}

\begin{document}

%opening
\title[Simplices in Poisson-Delaunay tessellations]{The volume of simplices in high-dimensional Poisson-Delaunay tessellations}

\author[Anna Gusakova]{Anna Gusakova}
\address{Faculty of Mathematics, Ruhr University Bochum, Germany}
\thanks{AG is supported by the the Deutsche Forschungsgemeinschaft (DFG) via RTG 2131 \textit{High-dimensional Phenomena in Probability -- Fluctuations and Discontinuity}.}
\email{anna.gusakova@rub.de}

\author[Christoph Th\"ale]{Christoph Th\"ale}
\address{Faculty of Mathematics, Ruhr University Bochum, Germany}
\thanks{CT is supported by by the DFG Scientific Network \textit{Cumulants, Concentration and Superconcentration}.}
\email{christoph.thaele@rub.de}


\begin{abstract}
Typical weighted random simplices $Z_{\mu}$, $\mu\in(-2,\infty)$, in a Poisson--Delaunay tessellation in $\mathbb{R}^n$ are considered, where the weight is given by the $(\mu+1)$st power of the volume. As special cases this includes the typical ($\mu=-1$) and the usual  volume-weighted ($\mu=0$) Poisson--Delaunay simplex. By proving sharp bounds on cumulants it is shown that the logarithmic volume of $Z_{\mu}$ satisfies a central limit theorem in high dimensions, that is, as $n\to\infty$. In addition, rates of convergence are provided. In parallel, concentration inequalities as well as moderate deviations are studied. The set-up allows the weight $\mu=\mu(n)$ to depend on the dimension $n$ as well. A number of special cases are discussed separately. For fixed $\mu$ also mod-$\phi$ convergence and the large deviations behaviour of the logarithmic volume of $Z_{\mu}$ are investigated.
\end{abstract}

\maketitle

%\tableofcontents

\section{Introduction and Selected results}

Establishing probabilistic limit theorems for convex bodies in high dimensions has been and is still one of the driving forces in the branch of mathematics known as Asymptotic Geometric Analysis. The arguably most prominent example is Klartag's central limit theorem \cite{KlartagCLT}. It roughly says that most lower-dimensional marginals of isotropic convex bodies are close to Gaussian distributions. More recently, other functionals and set-ups of high dimensional convex bodies were also investigated. For example, a central limit theorem for the volume of $k$-dimensional random projections of the $n$-dimensional cube with $k$ fixed and $n\to\infty$ is the content of the work of Paouris, Pivovarov and Zinn \cite{PPZ}, and for $k=1$ also moderate and large deviations were investigated \cite{KPT19a,KPT19b}. In addition, several types of limit theorems for the norms of (projections of) random points in the classical $\ell_p^n$-ball were proved by Gantert, Kim and Ramanan \cite{GKR}, Alonso-Guti\'errez, Prochno and Th\"ale \cite{APT18,APT17} and Kabluchko, Prochno and Th\"ale \cite{KPT19a,KPT19b}.

In high dimensional Stochastic Geometry, the probabilistic behaviour of \textit{random} convex sets is investigated when the dimension tends to infinity. For example, generalizing earlier works of Mathai \cite{Mathai82} and Ruben \cite{Ruben77}, Grote, Kabluchko and Th\"ale \cite{GKT17} investigated the logarithmic volume of a class of random simplices in high dimensions. Among other results, they proved that the logarithmic volume of the simplex generated by $n+1$ independent and uniformly distributed random points in the $n$-dimensional Euclidean unit ball satisfies a central limit theorem, as $n\to\infty$. For random simplices (and more general convex bodies) generated by product distributions with sub-exponential tails a similar result was derived by Alonso-Guti\'errez et al.\ \cite{Oberwolfach}. The present paper continues this line of research, but the model we will work with is rather different in the following sense. While the previous works \cite{Oberwolfach,GKT17,Mathai82,Ruben77} deal with a single random simplex, we work with an infinite collection of random simplices and then apply a probabilistic selection procedure to pick one of the simplices. To explain our set-up, we let $\eta$ be a stationary Poisson point process in $\R^n$ with intensity $\gamma\in(0,\infty)$, which might or might not depend on $n$. We now construct the Delaunay tessellation (also called Dalaunay triangulation) in $\R^n$ based on $\eta$. It gives rise to a dissection of $\R^n$ into an almost surely countable collection $\mathscr{D}$ of random simplices having the property that no point of $\eta$ is inside the circumball of any simplex in $\mathscr{D}$. We remark that the tessellation of $\R^n$ induced by $\mathscr{D}$ is dual to the well-known Poisson-Voronoi tessellation, see \cite{SW08}. Next, we select randomly one of the simplices from $\mathscr{D}$ in such a way that each simplex has the same chance of being selected. Intuitively, one can think of restricting $\mathscr{D}$ to the almost surely finite sub-collection of simplices that are contained in a `very large' ball and then selecting one of these simplices according to the uniform distribution, that is, regardless of its size and shape. We shall describe in Section \ref{sec:Del} how to make this construction mathematically rigorous using the notion of Palm distributions. We denote by $Z$ the outcome of the selection, which is (up to a random shift) known as the typical Poisson-Delaunay simplex in the stochastic geometry literature. In this paper we are interested in the probabilistic behaviour of the logarithmic $n$-volume
$$
Y_n:=\log V_n(Z)
$$
of $Z$ in high dimension, that is, as $n\to\infty$. A particular instance of our result we prove in this paper reads as follows. It provides a formula for the asymptotic behaviour of the expectation and the variance of the random variables $Y_n$ (part (i)) and describes their fluctuations in high dimensions (part (ii)) as well as their large deviations behaviour (part (iii)). We recall that $\gamma$ is the intensity of the underlying Poisson point process.

\begin{theorem}\label{thm:Intro}
\begin{itemize}
\item[(i)] As $n\to\infty$, 
$$
\Ex Y_n = -{n\over 2}\log n -\log\gamma +O(n)\qquad\text{and}\qquad\Var Y_n = {1\over 2}\log n+O(1).
$$

\item[(ii)] For a standard Gaussian random variable $N$ one has that
\begin{equation}\label{eq:CLTintro}
\sup_{x\in\R}\Big|\Pr\Big({Y_n-\Ex Y_n\over\sqrt{\Var Y_n}}\leq x\Big)-\Pr(N\leq x)\Big| \leq {c\over\sqrt{\log n}}
\end{equation}
for any $n\geq 3$, where $c\in(0,\infty)$ is an absolute constant. 

\item[(iii)] Define
\[
m_n:=-{n\over 2}\log n-{n\over 2}(\log\pi +1)+{7\over 4}\log n -\log\gamma.
\] 
Then, for each $x\in(0,\infty)$,
$$
\lim_{n\to\infty}{1\over{1\over 2}\log\big({n\over 2}\big)}\log\Pr(Y_n-m_n\geq x) = -{x^2\over 2}.
$$
\end{itemize}
\end{theorem}

In particular, letting $n\to\infty$ in \eqref{eq:CLTintro}, we conclude the convergence in distribution
$$
\widetilde{Y}_{n}:={Y_n-\Ex Y_n\over\sqrt{\Var Y_n}} \overset{D}{\longrightarrow} N.
$$
In addition to this central limit theorem we are able to prove an exponential concentration inequality, a Cram\'er-Petrov-type estimate and a Donsker-Varadhan-type moderate deviations principle for the centred and normalized logarithmic volume $\widetilde{Y}_{n}$. Moreover, we prove mod-Gaussian convergence and a large deviations principle for the sequence of random variables $Y_n$ after suitable shift and rescaling. In particular, these results cover the last part of Theorem \ref{thm:Intro}. As already emphasized above, part (ii) of Theorem \ref{thm:Intro} is only a special case of our Central Limit Theorem \ref{thm:Main}. In fact, we are able to deal with other procedures that select simplices from the infinite collection $\mathscr{D}$, where each simplex $c\in\mathscr{D}$ gets weight $V_n(c)^{\mu+1}$ for some $\mu\in(-2,\infty)$ and where a simplex $Z_{\mu}$ is now selected randomly according to these weights. Even more generally, we will allow $\mu$ to vary with $n$ and especially deal with the cases when $\mu$ is fixed or when $\mu=n^\alpha$ for some $\alpha\in(0,1)$, $\mu=\alpha n$ for some $\alpha\in(0,\infty)$ and $n-\mu=o(n)$. On the other hand, mod-Gaussian convergence and the large deviations principle for $\widetilde{Y}_n$ as $n\to\infty$ will be derived only for fixed weight parameter $\mu$.

The proof of Theorem \ref{thm:Intro} (ii) and its generalization, Theorem \ref{thm:Main}, relies on the general limit theory for large deviations of Saulis and Statulevi\v{c}ius \cite{SS91}, which in turn is based on sharp bounds for cumulants. This method has been intensively used during the last decade to derive a number of limit theorems for various quantities. Examples include determinants of random Wigner and block Hankel matrices \cite{DT19,DE13a}, spectral statistics of orthogonal polynomial ensembles \cite{PWZ17}, patterns in random permutations \cite{Hofer17}, weighted dependency graphs \cite{FerayWeightedDependency}, subgraph counts in the Erd\H{o}s-R\'enyi random graph \cite{DE13b,ModPhiBook}, stabilizing functionals in geometric probability \cite{BYY,ERS}, the volume fraction of a Boolean model and of a Poisson cylinder process \cite{Heinrich,HS09}, functionals of random polytopes \cite{GT18a,GT18b}, the volume of random simplices \cite{GKT17} as well as multiple stochastic integrals \cite{ST16}. The method is also well adapted to our situation, since we have directly access to the moment generating function of the random variables under consideration. This also allows for rather direct proofs of Theorem \ref{thm:Intro} (i) and (iii), where the latter relies in addition on an asymptotic analysis of the Barnes $G$-function. Let us remark that while the moments of integer order of the random variables $Y_n$ are known (see \cite{SW08}), we need an extension to non-integer moments as well as a generalization to the weighted random simplices $Z_{\mu}$. We develop the corresponding results in the present paper building on earlier works on beta random polytopes \cite{My4,KTT17}. This in turn also allows neat probabilistic interpretations of $Y_n$ in terms of independent gamma and beta random variables.

\medskip

The remaining parts of this paper are structured as follows. In Section \ref{sec:Del} we formally introduce the weighted random simplices $Z_{\mu}$ and derive an explicit formula for the moment generating function of the logarithmic volume of $Z_{\mu}$. We also present there our probabilistic interpretations for $Y_n$ and its weighted generalization. Section \ref{sec:polygamma} is devoted to a number of specific bounds and asymptotic expansions for polygamma functions. They are needed in Section \ref{sec:CumulantsResults}, where we derive sharp cumulant bounds for the logarithmic volume of $Z_{\mu}$. We also present and prove there our main results, Theorem \ref{thm:Main}, which as a special case includes Theorem \ref{thm:Intro} (ii). Section 5 is devoted to the results on mod-$\phi$ convergence and the large deviation principle for the logarithmic volume of $Z_{\mu}$ when $\mu$ is fixed. This includes a proof of Theorem \ref{thm:Intro} (iii).

\section{Weighted simplices in Poisson-Delaunay tessellations}\label{sec:Del}

\subsection{Description of the model}

Let $n\in\N$ and $\eta$ be a stationary Poisson point process on $\R^n$ with intensity $\gamma\in(0,\infty)$. For a $(n+1)$-tuple $(x_0,\ldots,x_n)$ of distinct points of $\eta$ we denote by $B(x_0,\ldots,x_n)$ the almost surely uniquely determined ball having the points $x_0,\ldots,x_n$ on its boundary. The points $x_0,\ldots,x_n$ then form a Delaunay simplex $\conv(x_0,\ldots,x_n)$ whenever $B(x_0,\ldots,x_n)$ does not contain any further point from $\eta$, that is, if $B(x_0,\ldots,x_n)\cap\eta=\{x_0,\ldots,x_n\}$. The collection $\mathscr{D}$ of all Delaunay simplices is called the Poisson-Delaunay tessellation of $\R^n$, see \cite[Chapter 10.2]{SW08}.

Our next goal is to describe a procedure to randomly select simplices from the tessellation $\mathscr{D}$, i.e., from the almost surely infinite collection of simplices $\mathscr{D}$.  For this purpose we introduce a parameter $\mu\in(-2,\infty)$. Moreover, if $c\in \mathscr{D}$ is a Delaunay simplex with vertices $x_0,\ldots,x_n\in\eta$ we shall write $z(c)$ for the midpoint of the ball $B(x_0,\ldots,x_n)$, that is, the circumcenter of $x_0,\ldots,x_n$. The set of all simplices in $\R^n$ is denoted by $\Simpl_n$. Endowing $\Simpl_n$ with the usual Hausdorff distance, we can define on $\Simpl_n$ the Borel $\sigma$-field $\mathcal{B}(\Simpl_n)$. Writing $\Simpl_n^0:=\{c\in\Simpl_n:z(c)=0\}$ we define a probability measure $\Pr_\mu^0$ on $\Simpl_n^0$ by 
\begin{equation}\label{eq:DefWeightedCellIndicator}
\Pr_\mu^0(A) := {1\over\gamma_\mu}\,\Ex\sum_{c\in \mathscr{D}\atop z(c)\in[0,1]^n}{\bf 1}\{c-z(c)\in A\}\,V_n(c)^{\mu+1},\qquad A\in\mathcal{B}(\Simpl_n^0),
\end{equation}
where $V_n(c)$ stands for the volume of $c$ and where 
$$
\gamma_\mu := \Ex\sum_{c\in \mathscr{D}\atop z(c)\in[0,1]^n}V_n(c)^{\mu+1}
$$
is normalizing constant. By $Z_\mu$ we denote a random simplex with distribution $\Pr_\mu^0$. Using this notation \eqref{eq:DefWeightedCellIndicator} can equivalently be expressed by saying that
\begin{equation}\label{eq:DefWeightedCellWithH}
\Ex h(Z_\mu) = {1\over\gamma_\mu}\,\Ex\sum_{c\in \mathscr{D}\atop z(c)\in[0,1]^n}h(c-z(c))\,V_n(c)^{\mu+1}
\end{equation}
for any non-negative measurable function $h:\Simpl_n^0\to\R$. This construction combines two particularly important special cases. Namely, if $\mu=-1$ then $\Pr_{-1}^0$ is just the distribution of the typical Delaunay simplex (with respect to the circumcenter as centring function) and if $\mu=0$ then $\Pr_{0}^0$ is the distribution of the typical volume-weighted Delaunay simplex (again with respect to the circumcenter as centring function) in the usual sense of Palm theory, see \cite{SW08}. It is well known that up to a random shift $Z_0$ coincides in distribution with the almost surely uniquely determined simplex from $\mathscr{D}$ containing the origin of $\R^n$, cf.\ \cite[Theorem 10.4.1]{SW08}.

We shall now establish a link between the distributions $\Pr_\mu^0$ and $\Pr_{-1}^0$.

\begin{lemma}\label{lem:LinkWeightedTypical}
For all non-negative measurable functions $h:\Simpl_n^0\to\R$ we have
$$
\Ex h(Z_\mu) = {1\over\Ex V_n(Z_{-1})^{\mu+1}}\,\Ex[h(Z_{-1})\,V_n(Z_{-1})^{\mu+1}].
$$
\end{lemma}
\begin{proof}
We can consider $\mathscr{D}$ as a stationary particle process of simplicial particles in $\R^n$ in the sense of \cite[Chapter 4]{SW08}. Each simplex $c\in\Simpl_n$ can be identified with a pair $(c_0,z)\in\Simpl_n^0\times\R^n$, where $c_0=c-z(c)$ and $z=z(c)$. Since $\mathscr{D}$ is a stationary particle process, its intensity measure $\Lambda$ admits the factorization $\Lambda=\gamma_{-1}\,\Pr_{-1}^0\otimes\lambda_n$ according to \cite[Theorem 4.1.1]{SW08}, where $\lambda_n$ stands for the Lebesgue measure on $\R^n$. Applying now Campbell's theorem (Theorem 3.1.2 in \cite{SW08}) and the translation invariance of $V_n$ we find that, for any non-negative measurable function $h:\Simpl_n^0\to\R$,
\begin{align*}
\gamma_\mu\,\Ex h(Z_\mu) &= \Ex\sum_{c\in\mathscr{D}\atop z(c)\in[0,1]^n}h(c-z(c))\,V_n(c)^{\mu+1}\\
&=\int\limits_{\Simpl_n}{\bf 1}\{z(c)\in[0,1]^n\}\,h(c-z(c))\,V_n(c)^{\mu+1}\,\Lambda(\dd c)\\
&=\gamma_{-1}\int\limits_{\Simpl_n^0}\int\limits_{\R^n}{\bf 1}\{z\in[0,1]^n\}\,h(c_0)\,V_n(c_0)^{\mu+1}\,\Pr_{-1}^0(\dd c_0)\,\lambda_n(\dd z)\\
&=\gamma_{-1}\,\Ex[h(Z_{-1})\,V_n(Z_{-1})^{\mu+1}].
\end{align*}
Choosing now $h\equiv 1$ we obtain that
$$
\gamma_\mu = \gamma_{-1}\,\Ex V_n(Z_{-1})^{\mu+1},
$$
which completes the argument.
\end{proof}

We introduce the notation
$$
\kappa_n:=\frac{\pi^{n/2}}{\Gamma\left(1+\frac{n}{2}\right)}
$$
for the volume of the $n$-dimensional unit ball and let $\sigma$ be the normalized spherical Lebesgue measure on the $(n-1)$-dimensional unit sphere $\Sp^{n-1}$. Furthermore, following the notation used in \cite{SW08} we define the quantities
$$
S(n,n,s) := \int\limits_{\Sp^{n-1}}\ldots\int\limits_{\Sp^{n-1}}\Delta_n(u_0,\ldots,u_n)^{s}\,\sigma(\dd u_0)\ldots\sigma(\dd u_n),\qquad s\in[0,\infty),
$$
and
$$
a_n:=\frac{n^2}{2^{n+1}\pi^{\frac{n-1}{2}}}\frac{\Gamma\big(\frac{n^2}{2}\big)}{\Gamma\big(\frac{n^2+1}{2}\big)}\left[\frac{\Gamma\left(\frac{n+1}{2}\right)}{\Gamma\left(1+\frac{n}{2}\right)}\right]^n,
$$
where $\Delta_n(u_0,\ldots,u_n):=V_n(\conv(u_0,\ldots,u_n))$. This notation allows us to identify the values of $\Ex V_n(Z_{-1})^{\mu+1}$.

\begin{lemma}\label{lem:GammaMu}
For any $s\in(-1,\infty)$ we have
\begin{align*}
\Ex V_n(Z_{-1})^{s} &= a_n\,S(n,n,s+1)\,{\Gamma(n+s)\over n\,\kappa_n^{n+s}}\,{1\over\gamma^{s}}.
\end{align*}
\end{lemma}
\begin{proof}
The moments $\Ex V_n(Z_{-1})^{s}$ are known from \cite[Theorem 10.4.5]{SW08} for integer values of $s$. However, the proof can be easily extended to non-integer values for $s$. This yields the result.
\end{proof}

In the next step, we develop a probabilistic representation for the distributions $\Pr_\mu^0$.

\begin{theorem}\label{thm:Representation}
For $\mu\in(-2,\infty)$ and $A\in\mathcal{B}(\Simpl_n^0)$ one has that
\begin{align*}
\Pr_\mu^0(A) &= {n\kappa_n^{n+\mu+1}\,\gamma^{n+\mu+1}\over S(n,n,\mu+2)\,\Gamma(n+\mu+1)}\int\limits_0^\infty\int\limits_{(\Sp^{n-1})^{n+1}}{\bf 1}\{\conv(ru_0,\ldots,ru_n)\in A\}\\
&\qquad\times e^{-\gamma\kappa_nr^n}\,r^{n^2+n(\mu+1)-1}\,\Delta_n(u_0,\ldots,u_n)^{\mu+2}\,\sigma^{n+1}(\dd(u_0,\ldots,u_n))\,\dd r.
\end{align*}
\end{theorem}
\begin{proof}
According to \cite[Theorem 10.4.4]{SW08} we have that
\begin{align*}
\Pr_{-1}^0(A) & = a_n\gamma^n\int\limits_{0}^{\infty}\int\limits_{(\Sp^{n-1})^{n+1}}{\bf 1}\{\conv(ru_0,\ldots,ru_n)\in A\}\,e^{-\gamma\kappa_nr^n}r^{n^2-1}\\
&\qquad\qquad\qquad\qquad\times\Delta_n(u_0,\ldots,u_n)\,\sigma^{n+1}(\dd(u_0,\ldots,u_n))\,\dd r.
\end{align*}
Using now Lemma \ref{lem:LinkWeightedTypical} with $h(\cdot)={\bf 1}\{\cdot\in A\}$ and Lemma \ref{lem:GammaMu} we conclude that
\begin{align*}
\Pr_{\mu}^0(A) &= {a_n\,\gamma^n\over\Ex V_n(Z_{-1})^{\mu+1}}\int\limits_{0}^{\infty}\int\limits_{(\Sp^{n-1})^{n+1}}{\bf 1}\{\conv(ru_0,\ldots,ru_n)\in A\}\,e^{-\gamma\kappa_nr^n}r^{n^2-1}\\
&\qquad\qquad\times\Delta_n(u_0,\ldots,u_n)\,\Delta_n(ru_0,\ldots,ru_n)^{\mu+1}\,\sigma^{n+1}(\dd(u_0,\ldots,u_n))\,\dd r\\
&={n\,\kappa_n^{n+\mu+1}\,\gamma^{n+\mu+1}\over S(n,n,\mu+2)\,\Gamma(n+\mu+1)}\int\limits_{0}^{\infty}\int\limits_{(\Sp^{n-1})^{n+1}}{\bf 1}\{\conv(ru_0,\ldots,ru_n)\in A\}\,e^{-\gamma\kappa_nr^n}\\
&\qquad\qquad\times r^{n^2+n(\mu+1)-1}\,\Delta_n(u_0,\ldots,u_n)^{\mu+2}\,\sigma^{n+1}(\dd(u_0,\ldots,u_n))\,\dd r.
\end{align*}
This completes the proof of the theorem.
\end{proof}

As a first corollary of Theorem \ref{thm:Representation} we derive a probabilistic representation for the radius of the circumsphere of the random simplex $Z_{\mu}$.
Given a simplex $c\in \Simpl_n$ denote by $R(c)$ the radius of the circumsphere of $c$. The next lemma provides the formula for the cumulative distribution function $F_{\mu}(t):=\Pr(R(Z_\mu)\leq t)$, $t> 0$, of the random variable $R(Z_\mu)$. We remark that for the typical Delaunay simplex $Z_{-1}$ this formula was known before from \cite{ENR17}.

\begin{corollary}\label{lm:RadiiDistribution}
For $\mu\in(-2,\infty)$ and $t>0$ one has that
\[
F_{\mu}(t)={n\kappa_n^{n+\mu+1}\gamma^{n+\mu+1}\over \Gamma(n+\mu+1)}\int\limits_{0}^{t}e^{-\gamma\kappa_nr^n}r^{n^2-1+n(\mu+1)}\,\dd r.
\]
\end{corollary}

\begin{proof}
Applying Theorem \ref{thm:Representation} we get
\begin{align*}
\Pr(R(Z_\mu)\leq t)&=\Ex\left[{\bf 1}\{R(Z_\mu)\leq t\}\right]\\
&={n\kappa_n^{n+\mu+1}\,\gamma^{n+\mu+1}\over S(n,n,\mu+2)\,\Gamma(n+\mu+1)}\int\limits_0^\infty\int\limits_{(\Sp^{n-1})^{n+1}}{\bf 1}\{R(\conv(ru_0,\ldots,ru_n))\leq t\}\\
&\qquad\times e^{-\gamma\kappa_nr^n}\,r^{n^2+n(\mu+1)-1}\,\Delta_n(u_0,\ldots,u_n)^{\mu+2}\,\sigma^{n+1}(\dd(u_0,\ldots,u_n))\,\dd r.
\end{align*}
Using Fubini's theorem and the definition of $S(d,d,\mu+2)$ we conclude
\begin{align*}
\Pr(R(Z_\mu)\leq t)&={n\kappa_n^{n+\mu+1}\,\gamma^{n+\mu+1}\over S(n,n,\mu+2)\,\Gamma(n+\mu+1)}\int\limits_0^te^{-\gamma\kappa_nr^n}\,r^{n^2+n(\mu+1)-1}\dd r\\
&\qquad\times \int\limits_{(\Sp^{n-1})^{n+1}}\,\Delta_n(u_0,\ldots,u_n)^{\mu+2}\,\sigma^{n+1}(\dd(u_0,\ldots,u_n))\\
&={n\kappa_n^{n+\mu+1}\gamma^{n+\mu+1}\over \Gamma(n+\mu+1)}\int\limits_{0}^{t}e^{-\gamma\kappa_nr^n}r^{n^2-1+n(\mu+1)}\,\dd r
\end{align*}
and complete the proof of the lemma
\end{proof}

\subsection{Moments of the volume of the random simplex $Z_\mu$}

As a second application of Theorem \ref{thm:Representation} we compute the moments of the random variable $V_n(Z_\mu)$.

\begin{theorem}\label{cor:MomentsAbstract}
For $\mu\in(-2,\infty)$ and $s\in(-\mu-2,\infty)$ we have that
\begin{align*}
\Ex V_n(Z_\mu)^s &= {S(n,n,\mu+s+2)\over S(n,n,\mu+2)}{\Gamma(n+\mu+s+1)\over\Gamma(n+\mu+1)}{1\over(\gamma\,\kappa_n)^s}.
\end{align*}
\end{theorem}
\begin{proof}
Using Theorem \ref{thm:Representation} we get
\begin{align*}
\Ex V_n(Z_\mu)^s &= {n\kappa_n^{n+\mu+1}\,\gamma^{n+\mu+1}\over S(n,n,\mu+2)\,\Gamma(n+\mu+1)}\int\limits_0^\infty\int\limits_{(\Sp^{n-1})^{n+1}}\Delta_n(ru_0,\ldots,ru_n)^s\\
&\qquad\times e^{-\gamma\kappa_nr^n}\,r^{n^2+n(\mu+1)-1}\,\Delta_n(u_0,\ldots,u_n)^{\mu+2}\,\sigma^{n+1}(\dd(u_0,\ldots,u_n))\,\dd r\\
&={n\kappa_n^{n+\mu+1}\,\gamma^{n+\mu+1}\over S(n,n,\mu+2)\,\Gamma(n+\mu+1)}\int\limits_0^\infty\int\limits_{(\Sp^{n-1})^{n+1}}e^{-\gamma\kappa_nr^n}\,r^{n^2+n(\mu+s+1)-1}\\
&\qquad\qquad\qquad\times\Delta_n(u_0,\ldots,u_n)^{\mu+s+2}\,\sigma^{n+1}(\dd(u_0,\ldots,u_n))\,\dd r\\
&={n\kappa_n^{n+\mu+1}\,\gamma^{n+\mu+1}\over S(n,n,\mu+2)\,\Gamma(n+\mu+1)}\,S(n,n,\mu+s+2)\,\int\limits_0^\infty e^{-\gamma\kappa_nr^n}\,r^{n^2+n(\mu+s+1)-1}\,\dd r\\
&={n\kappa_n^{n+\mu+1}\,\gamma^{n+\mu+1}\over S(n,n,\mu+2)\,\Gamma(n+\mu+1)}\,S(n,n,\mu+s+2)\,{\Gamma(n+\mu+s+1)\over n\,(\gamma\kappa_n)^{n+\mu+s+1}}\\
&={S(n,n,\mu+s+2)\over S(n,n,\mu+2)}{\Gamma(n+\mu+s+1)\over\Gamma(n+\mu+1)}{1\over(\gamma\,\kappa_n)^s}\,.
\end{align*}
This completes the proof.
\end{proof}

The precise values for $S(n,n,s)$, $s\in[0,\infty)$, can be found for integer values of $s$ in \cite{rM71} and for general $s$ in \cite{KTT17}.
%and \cite{My4}. 
In particular, from an analysis of the expected volume of so-called random beta-polytopes (see \cite[Proposition 2.8]{KTT17}) we obtain, by formally putting $\beta=-1$ there, 
\begin{align*}
S(n,n,s) = \frac{2^{n+1}\pi^{\frac{n(n+1)}{2}}\Gamma\left(\frac{n^2+n(s-1)+s}{2}\right)}{(n!)^s\Gamma\left(\frac{n^2+n(s-1)}{2}\right)\Gamma\left(\frac{n+s}{2}\right)^{n+1}}\prod\limits_{i=1}^{n}\frac{\Gamma\left(\frac{s+i}{2}\right)}{\Gamma\left(\frac{i}{2}\right)}
\end{align*}
and hence from Theorem \ref{cor:MomentsAbstract} we have
\begin{equation}\label{eq:Moments}
\begin{aligned}
\Ex V_n(Z_\mu)^s =C_{\mu}\times\left[\frac{\Gamma\left(\frac{n}{2}+1\right)}{\gamma\pi^{n/2}n!}\right]^s&\times{\Gamma\left({(n+1)(n+\mu)\over 2}+1+{n+1\over 2}s\right)\over\Gamma\big({n(n+\mu+1)\over 2}+{n\over 2}s\big)}{\Gamma(n+\mu+1+s)\over\Gamma\big({n+\mu\over 2}+1+{s\over 2}\big)^{n+1}}\\
&\times\prod\limits_{i=1}^n\Gamma\left({i+\mu\over 2}+1+{s\over 2}\right),
\end{aligned}
\end{equation}
where the constant $C_\mu$ is given by
\[
C_{\mu}:={\Gamma\big({n^2+n(\mu+1)\over 2}\big)\Gamma\big({n+\mu\over 2}+1\big)^{n+1}\over \Gamma(n+\mu+1)\Gamma\big({n^2+n(\mu+1)+\mu\over 2}+1\big) \prod\limits_{i=1}^n\Gamma\big({i+\mu\over 2}+1\big)}.
\]

\subsection{Probabilistic representations}

Knowing the moments of the volume of random simplex $Z_{\mu}$ we are able to say more about its actual distribution, namely we can provide a neat probabilistic representation for the random variable $V_n(Z_{\mu})^2$, which is similar in spirit to the ones for Gaussian or beta random simplices \cite{GKT17,rM71}. Before we formulate our result let us recall some standard distributions.

A random variable has a Gamma distribution with shape $\alpha\in(0,\infty)$ and rate $\lambda \in(0,\infty)$ if its density function is given by
\[
g_{\alpha,\lambda}(t)={\lambda^{\alpha}\over\Gamma(\alpha)}t^{\alpha-1}e^{-\lambda t},\quad t\in(0,\infty).
\]
A random variable has a Beta distribution with shape parameters $\alpha, \beta\in(0,\infty)$ if its density function is given by
\[
g_{\alpha,\beta}(t)={\Gamma(\alpha+\beta)\over\Gamma(\alpha)\Gamma(\beta)}t^{\alpha-1}(1-t)^{\beta - 1},\quad t\in(0,1).
\]
We will use the notation $\xi\sim \Gam(\alpha,\lambda)$ and $\xi\sim \Bet(\alpha,\beta)$ to indicate that random variable $\xi$ has Gamma distribution with shape $\alpha$ and rate $\lambda$ or Beta distribution with shape parameters $\alpha, \beta$, respectively. Moreover, $\xi\overset{D}{=}\xi'$ will indicate that two random variables $\xi$ and $\xi'$ have the same distribution.

\begin{theorem}\label{thm:ProbabilisticRepresentation}
For any $\mu\in(-2,\infty)$ we have that
\[
\xi^n(1-\xi)\left(n!V_n(Z_{\mu})\right)^2\eqdistr \left({\rho\over\gamma\kappa_n}\right)^2\prod\limits_{i=1}^n\xi_i,
\]
where $\xi\sim\Bet\left({n^2+n+n\mu\over 2}, {\mu+2\over 2}\right)$, $\xi_i\sim\Bet\left({i+\mu+1\over 2}, {n-i+1\over 2}\right)$, $i\in\{1,\ldots,n\}$, and $\rho\sim\Gam(n+\mu+1,1)$ are independent random variables, independent of $V_n(Z_{\mu})$.
\end{theorem}

\begin{proof}
First of all let us recall, that for $\xi_i\sim\Bet\left({i+\mu+1\over 2}, {n-i+1\over 2}\right)$ and $s\in(0,\infty)$ we have
\[
\Ex[\xi_i^{s}]={\Gamma\left({n+\mu\over 2}+1\right)\Gamma\left({i+\mu+1\over 2}+s\right)\over \Gamma\left({i+\mu+1\over 2}\right)\Gamma\left({n+\mu\over 2}+1+s\right)},
\]
and for $\rho\sim\Gam(n+\mu+1,1)$ and $s\in(0,\infty)$ we have
\[
\Ex[\rho^{2s}]={\Gamma(s+n+\mu+1)\over \Gamma\left(n+\mu+1\right)}.
\]
Moreover, for $s\in(0,\infty)$ we compute that
\begin{align*}
\Ex\left[\xi^{ns}(1-\xi)^s\right]&={\Gamma\left({(n+1)(n+\mu)\over 2}+1\right)\over\Gamma\left({n^2+n+n\mu\over 2}\right)\Gamma\left({\mu+2\over 2}\right)}\int\limits_{0}^1t^{{n^2+n+n\mu+ns\over 2}-1}(1-t)^{{\mu\over 2} +s}dt\\
&={\Gamma\left({(n+1)(n+\mu)\over 2}+1\right)\Gamma\left({n(n+1+\mu)\over 2}+ns\right)\Gamma\left({\mu\over 2}+1+s\right)\over\Gamma\left({n^2+n+n\mu\over 2}\right)\Gamma\left({\mu+2\over 2}\right)\Gamma\left({(n+1)(n+\mu)\over 2}+1+(n+1)s\right)}.
\end{align*}
Combining this with \eqref{eq:Moments} we conclude that, for all $s\in(0,\infty)$,
\[
(\gamma\kappa_n n!)^{2s}\,\Ex\left[\xi^{sn}(1-\xi)^sV_n(Z_{\mu})^{2s}\right]=\Ex\left[\rho^{2s}\prod\limits_{i=1}^n\xi_i^s\right],
\]
which finishes the proof.
\end{proof}

Using the formula for the cumulative distribution function of the radius $R(Z_{\mu})$ of the circumsphere of $Z_{\mu}$ provided in Corollary \ref{lm:RadiiDistribution} we also obtain the following probabilistic equality.

\begin{proposition}\label{lm:ProbabilisticReprRadii}
For any $\mu\in(-2,\infty)$ we have
\[
R(Z_{\mu})^n\eqdistr {\rho\over \gamma\kappa_n},
\]
where $\rho\sim\Gam(n+\mu+1,1)$.
\end{proposition}
\begin{proof}
Using Corollary \ref{lm:RadiiDistribution} we obtain
\begin{align*}
\Pr(\gamma\kappa_nR(Z_\mu)^n\leq t)&=\Pr\left(R(Z_\mu)\leq {t^{1/n}\over \gamma\kappa_n}\right)\\
&={n\kappa_n^{n+\mu+1}\gamma^{n+\mu+1}\over \Gamma(n+\mu+1)}\int\limits_{0}^{{t^{1/n}/ \gamma\kappa_n}}e^{-\gamma\kappa_nr^n}r^{n^2-1+n(\mu+1)}\,\dd r\\
&={1\over \Gamma(n+\mu+1)}\int\limits_{0}^{t}e^{-y}y^{n+\mu}\,\dd y\\
&=\Pr(\rho\leq t).
\end{align*}
This completes the proof.
\end{proof}

Let us point out, that the random variable ${\rho\over \gamma\kappa_n}$ also appears in the right side of the equality in Theorem \ref{thm:ProbabilisticRepresentation}. This is an evidence for the fact that the distribution of the vertices of the random simplex $Z_{\mu}$ is a randomly rescaled distribution on the unit sphere. The next result specifies this distribution for integer values of $\mu$.

\begin{proposition}
For any integer $\mu\in\{-1,0,1,2,\ldots\}$ we have
\[
\xi^{n}\,V_n(Z_{\mu})^2\eqdistr \left({\rho\over \gamma\kappa_n}\right)^{2}\V_n\left(\conv(X_0,\ldots, X_n)\right)^2,
\]
where $X_0,\ldots, X_n$ are independent and distributed uniformly on the unit $(n+\mu+2)$-dimensional sphere, $\xi\sim\Bet\left({n^2+n+n\mu\over 2}, {\mu+2\over 2}\right)$ is independent of $V_n(Z_{\mu})$ and $\rho\sim\Gam(n+\mu+1,1)$ is independent of $X_0,\ldots, X_n$.
\end{proposition}
\begin{proof}
From \cite[Theorem 2.3]{GKT17} we have
\begin{align*}
\Ex\left[\V_n\left(\conv(X_0,\ldots, X_n)\right)^{2s}\right]&={1\over n!^{2s}}\prod\limits_{i=1}^n{\Gamma\left({\mu+i\over 2}+1+s\right)\over\Gamma\left({\mu+i\over 2}+1\right)}{\Gamma\left({n+\mu\over 2}+1\right)^{n+1}\over \Gamma\left({n+\mu\over 2}+1+s\right)^{n+1}}\\
&\qquad\qquad\times{\Gamma\left({(n+1)(n+\mu)\over 2}+1+(n+1)s\right)\over \Gamma\left({(n+1)(n+\mu)\over 2}+1+ns\right)},
\end{align*}
for any $s\in(0,\infty)$. Using the equality
\[
\Ex[\xi^{ns}]={\Gamma\left({(n+1)(n+\mu)\over 2}+1\right)\Gamma\left({n^2+n+n\mu\over 2}+ns\right)\over \Gamma\left({n^2+n+n\mu\over 2}\right)\Gamma\left({(n+1)(n+\mu)\over 2}+1+ns\right)}
\]
and combining this with \eqref{eq:Moments} we conclude, for all $s\in(0,\infty)$, 
\[
\Ex\left[\xi^{ns}V_n(Z_{\mu})^{2s}\right]=\Ex\left[\left({\rho\over \gamma\kappa_n}\right)^{2s}\V_n\left(\conv(X_0,\ldots, X_n)\right)^{2s}\right],
\]
which finishes the proof.
\end{proof}

\begin{remark}\rm 
In \cite[Theorem 2.7]{GKT17} it was shown, that the random variable $\xi$ in the previous proposition is equal by distribution to the squared distance from the origin to the $n$-dimensional affine subspace spanned by the random vectors $X_0,\ldots, X_n$. This is equivalent to say, that 
\[
1-\xi\eqdistr R\left(\conv(X_0,\ldots, X_n)\right)^2.
\]
\end{remark}

\section{Asymptotics for polygamma functions}\label{sec:polygamma}

Later in this paper we will need to analyse sums of polygamma functions as the number of terms tends to infinity. In this section we derive some useful identities allowing to identify their asymptotic behaviour.

We will use the following notation.
%Given two sequences of real numbers $(a_n)_{n\in\mathbb{N}}$ and $(b_n)_{n\in\mathbb{N}}$ we write $a_n \sim b_n$ if $a_n/b_n\to 1$, as $n\to\infty$. Moreover g
Given two functions $f(x)$ and $g(x)$ we write $f=O(g)$ if $\limsup\limits_{x\to\infty}|f(x)/g(x)|<\infty$. Moreover, we write $f=o(g)$ if $\lim\limits_{x\to\infty}|f(x)/g(x)|=0$.

Before we start let us recall some well-known asymptotic results for gamma and polygamma functions. The first of them is Stirling's formula \cite{NIST}:
\begin{align}
\log\Gamma(x)&=x\log x-x-\frac12\log x+O(1),\qquad \text{as}\  x\rightarrow\infty,\label{eq:AsympGamma}\\
\log(n!)&=n\log n-n+\frac12 \log n+O(1),\qquad\text{as}\ n\to\infty,\, n\in\mathbb{N}.\label{eq:Stirling}
\end{align}
Consider the digamma function $\psi(x)=\psi^{(0)}(x):=\frac{\dd}{\dd x}\log\Gamma(x)$ and, more generally, for $m\in\mathbb{N}$ the polygamma function 
\[
\psi^{(m)}(x):=\frac{\dd^m}{\dd x^m}\psi(z)=\frac{\dd^{m+1}}{\dd x^{m+1}}\log\Gamma(x).
\]
In \cite[Theorem C]{QV04} it was shown that
\begin{equation}\label{eq:AsympDigamma}
\psi(x)=\log x-{1\over 2x}+O(1/x^{2})
\end{equation}
and in \cite{Mor10} the asymptotics
\begin{equation}\label{eq:AsympPolygamma}
\psi^{(1)}(x)={1\over x}+{1\over 2x^2}+O(1/x^{3})
\end{equation}
was obtained, as $x\rightarrow\infty$. Moreover, for any $x\neq 0,-1,-2,\ldots$ one has that
\begin{equation}\label{eq:PolygammaSeriesFormula}
\psi^{(m)}(x)=\sum\limits_{k=0}^{\infty}\frac{(-1)^{m+1}m!}{(x+k)^{m+1}}.
\end{equation}
Hence, according to \cite{AS64}, we conclude that
\begin{equation}\label{eq:PolygammaBound}
|\psi^{(m)}(x)|\leq \frac{(m-1)!}{x^m}+\frac{m!}{x^{m+1}}.
\end{equation}

The following three propositions provide identities or estimates for sums of polygamma functions. They will be important in the proof of our cumulant estimates provided in Section \ref{sec:CumulantsResults}, which in turn are the key to establish our main central limit theorem.

\begin{proposition}\label{prop:DigammaSum}
For any $a\in(0,\infty)$ and $k\in\N$, $k\ge 2$ we have
\begin{align*}
\frac12\sum\limits_{j=1}^k\psi\left({j+a\over 2}\right)&=
\left({k-c\over 2}+{a\over 2}-{1\over 2}\right)\psi (a+k-c-1)+{c\over 2}\psi\left(a+k-1\right)+{1\over 4}\psi\left({a+k\over 2}\right)\\
&-\left({a\over 2}-{1\over 2}\right)\psi(a+1)-{1\over 4}\psi\left({a\over 2}+1\right)-{k\over 2}\left(1+\log 2\right) +1+2c,
\end{align*}
where $c:= k\mod 2$, which is equal to $0$ if $k$ is even and equal to $1$ if $k$ is odd.
\end{proposition}

\begin{proof}
Put $p:=\lfloor{k\over 2}\rfloor$ and $c:= k\mod 2$. Legrendre's duplication formula (see \cite{NIST}) says that
$$
\Gamma(2x) = {2^{2x-1}\Gamma(x)\Gamma\left(x+{1\over 2}\right)\over\sqrt{\pi}},\qquad x\in(0,\infty).
$$
As a consequence,
\begin{align*}
{\dd\over \dd x}\log\Gamma(2x)  &={\dd\over\dd x}\Big[(2x-1)\log 2-{1\over 2}\log\pi+\log\Gamma(x)+\log\Gamma\Big(x+{1\over 2}\Big)\Big]\\
&=2\log 2+\psi(x)+\psi\Big(x+{1\over 2}\Big).
\end{align*}
Moreover, since ${\dd\over \dd x}\log\Gamma(2x)=2\psi(2x)$ we deduce that
\begin{equation}\label{eq:DoublingPsi}
{1\over 2}\left(\psi(x)+\psi\left(x+\frac12\right)\right)=\psi(2x)-\log 2.
\end{equation}
Using the above equality we have
\begin{align}
\frac12\sum\limits_{j=1}^k\psi\left({j+a\over 2}\right)&=\sum\limits_{j=1}^p{1\over 2}\left(\psi\left({2j-1+a\over 2}\right)+\psi\left({2j+a\over 2}\right)\right)+{c\over 2}\psi\left({k+a\over 2}\right)\notag\\
&=\sum\limits_{j=1}^p\psi(2j-1+a)-p\log 2+{c\over 2}\psi\left({k+a\over 2}\right).\label{eq:1}
\end{align}
Applying the identity $\psi(x+1)=\psi(x)+{1\over x}$, which follows from the definition of the function $\psi$ together with the fact that $\Gamma(x+1)=x\Gamma(x)$, $x\in(0,\infty)$, we get
\begin{align*}
\sum\limits_{j=1}^p\psi(2j-1+a)&=\sum\limits_{j=1}^p\sum\limits_{l=1}^{2j-2}{1\over a+l}+p\psi(a+1)\\
&=p\sum\limits_{l=1}^{2p-2}{1\over a+l}-\sum\limits_{j=1}^{p-1}j\left({1\over a+2j-1}+{1\over a+2j}\right)+p\psi(a+1)\\
&=\sum\limits_{j=1}^{p-1}{p-j\over a+2j-1}+\sum\limits_{j=1}^{p-1}{p-j\over a+2j}+p\psi(a+1)\\
&=\left(p+{a\over 2}-{1\over 2}\right)\sum\limits_{j=1}^{p-1}{1\over a+2j-1}+\left(p+{a\over 2}\right)\sum\limits_{j=1}^{p-1}{1\over a+2j}-p+1+p\psi(a+1)\\
&=\left(p+{a\over 2}-{1\over 2}\right)\sum\limits_{j=1}^{2p-2}{1\over a+j}+{1\over 4}\sum\limits_{j=1}^{p-1}{1\over {a\over 2}+j}-p+1+p\psi(a+1).
\end{align*}
Finally, using the equalities
\[
\sum\limits_{j=1}^{2p-2}{1\over a+j}=\psi(a+2p-1)-\psi(a+1)
\]
and
$$
\sum\limits_{j=1}^{p-1}{1\over {a\over 2}+j}=\psi\left({a\over 2}+p\right)-\psi\left({a\over 2}+1\right)
$$
we conclude that
\begin{align*}
\sum\limits_{j=1}^p\psi(2j-1+a)&=\left(p+{a\over 2}-{1\over 2}\right)\psi (a+2p-1)-\left({a\over 2}-{1\over 2}\right)\psi(a+1)\\
&+{1\over 4}\psi\left({a\over 2}+p\right)-{1\over 4}\psi\left({a\over 2}+1\right)-p +1.
\end{align*}
Substituting this into \eqref{eq:1} completes the proof.
\end{proof}

\begin{proposition}\label{prop:DerivativeDigammaSum}
For any $a\in(0,\infty)$ and $k\in\N$, $k\ge 2$ we have
\begin{align*}
&\frac14\sum\limits_{j=1}^k\psi^{(1)}\left({j+a\over 2}\right)\\
&={1\over 2}\left(\psi(a+k-c+1)-\psi(a+1)\right)+{a\over 2}\left(\psi^{(1)}(a+k-c+1)-\psi^{(1)}(a+1)\right)\\
&\quad -{1\over 8}\Big(\psi^{(1)}\Big({a+k-c+1\over 2}\Big)-\psi^{(1)}\Big({a+1\over 2}\Big)\Big)+{k-c\over 2}\psi^{(1)}(a+k-c+1)+{c\over 4}\psi^{(1)}\left({k+a\over 2}\right),
\end{align*}
where $c:= k\mod 2$.
\end{proposition}

\begin{proof}
We will again use the notation $p:=\lfloor{k\over 2}\rfloor$ and $c:= k\mod 2$. Applying \eqref{eq:PolygammaSeriesFormula} we have that
\begin{align}
&\frac14\sum\limits_{j=1}^k\psi^{(1)}\left({j+a\over 2}\right)\notag=\sum\limits_{i=1}^k\sum\limits_{j=0}^{\infty}{1\over (i+a+2j)^2}\notag\\
&=\sum\limits_{i=1}^p\sum\limits_{j=0}^{\infty}{1\over (2i+a+2j)^2}+\sum\limits_{i=1}^p\sum\limits_{j=0}^{\infty}{1\over (2i+a+2j-1)^2}+{c\over 4}\psi^{(1)}\left({k+a\over 2}\right).\label{eq:2}
\end{align}
For $b\in(0,\infty)$ consider the following expression:
$$
T(b):=\sum\limits_{i=1}^p\sum\limits_{j=i}^{\infty}{1\over (2j+b)^2}.
$$
Then
\begin{align*}
T(b)&=\sum\limits_{i=1}^p\sum\limits_{j=i}^{p}{1\over (2j+b)^2}+\sum\limits_{i=1}^p\sum\limits_{j=p+1}^{\infty}{1\over (2j+b)^2}\\
&=\frac12\sum\limits_{i=1}^p{2i\over (2i+b)^2}+p\sum\limits_{j=p+1}^{\infty}{1\over (2j+b)^2}\\
&=\frac12\sum\limits_{i=1}^p{1 \over 2i+b}-{b\over 2}\sum\limits_{i=1}^p{1\over (2i+b)^2}+p\sum\limits_{j=p+1}^{\infty}{1\over (2j+b)^2}\\
&=\frac14\sum\limits_{i=1}^p{1 \over i+{b\over 2}}-{b\over 8}\sum\limits_{j=0}^\infty{1\over \left(i+1+{b\over 2}\right)^2}+\left({p\over 4}+{b\over 8}\right)\sum\limits_{j=0}^{\infty}{1\over \left(j+{b\over 2}+p+1\right)^2}.
\end{align*}
Using now \eqref{eq:PolygammaSeriesFormula} we conclude that
\begin{align*}
T(b) &= {1\over 4}\psi\left({b\over 2}+p+1\right)-{1\over 4}\psi\left({b\over 2}+1\right)-{b\over 8}\psi^{(1)}\left({b\over 2}+1\right)+\left({p\over 4}+{b\over 8}\right)\psi^{(1)}\left({b\over 2}+p+1\right).
\end{align*}
Moreover, note that similarly to \eqref{eq:DoublingPsi} one has the duplication formula
$$
\psi^{(1)}(2x) = {1\over 4}\left(\psi^{(1)}(x)+\psi^{(1)}\left(x+{1\over 2}\right)\right),\qquad x\in(0,\infty).
$$
Using this together with \eqref{eq:DoublingPsi} and substituting the above representation for $T(b)$ into \eqref{eq:2} we get
\begin{align*}
&\frac14\sum\limits_{j=1}^k\psi^{(1)}\left({j+a\over 2}\right)=T(a)+T(a-1)+{c\over 4}\psi^{(1)}\left({k+a\over 2}\right)\\
&={1\over 2}\left(\psi(a+2p+1)-\psi(a+1)\right)-{a\over 2}\left(\psi^{(1)}(a+1)-\psi^{(1)}(a+2p+1)\right)\\
&\qquad -{1\over 8}\Big(\psi^{(1)}\Big({a+2p+1\over 2}\Big)-\psi^{(1)}\Big({a+1\over 2}\Big)\Big)+p\psi^{(1)}(a+2p+1)+{c\over 4}\psi^{(1)}\left({k+a\over 2}\right).
%\frac12\psi(a+2p+1)-\frac12\psi(a+1)-{a\over 8}\psi^{(1)}\left({a\over 2}+1\right)-{a-1\over 8}\psi^{(1)}\left({a+1\over 2}\right)\\
%&+\left({p\over 4}+{a\over 8}\right)\psi^{(1)}\left({a\over 2}+p+1\right)+\left({p\over 4}+{a-1\over 8}\right)\psi^{(1)}\left({a+1\over 2}+p\right)+{c\over 4}\psi^{(1)}\left({k+a\over 2}\right).
\end{align*}
This completes the proof upon replacing $2p$ by $k-c$.
\end{proof}

\begin{proposition}\label{prop:DerivativeDigammaSumBound}
For any $a\in(0,\infty)$ and $k,m\in\N$, $k,m\ge 2$ we have that 
\[
\left|{1\over 2^{m+1}}\sum\limits_{j=1}^k\psi^{(m)}\left({j+a\over 2}\right)\right|\leq{4m!\over (a+1)^{m-1}}.
\]
\end{proposition}

\begin{proof}
We proceed analogously to the proof of Proposition \ref{prop:DerivativeDigammaSum}.
Using the notation $p:=\lfloor{k\over 2}\rfloor$ and $c:= k\mod 2$, and applying \eqref{eq:PolygammaSeriesFormula} we write
\begin{align*}
{1\over 2^{m+1}}\sum\limits_{j=1}^k\psi^{(m)}\left({j+a\over 2}\right)&=\sum\limits_{i=1}^p\sum\limits_{j=0}^{\infty}{(-1)^{m+1}m!\over (2i+a+2j)^{m+1}}\\
&+\sum\limits_{i=1}^p\sum\limits_{j=0}^{\infty}{(-1)^{m+1}m!\over (2i+a+2j-1)^{m+1}}+{c\over 2^{m+1}}\psi^{(m)}\left({k+a\over 2}\right).
\end{align*}
Consider for $b\in(0,\infty)$ the following expression:
\begin{align*}
T(b):&=\sum\limits_{i=1}^p\sum\limits_{j=i}^{\infty}{(-1)^{m+1}m!\over (2j+b)^{m+1}}\\
&=\sum\limits_{i=1}^p\sum\limits_{j=i}^{p}{(-1)^{m+1}m!\over (2j+b)^{m+1}}+\sum\limits_{i=1}^p\sum\limits_{j=p+1}^{\infty}{(-1)^{m+1}m!\over (2j+b)^{m+1}}\\
&=\frac12\sum\limits_{i=1}^p{2i(-1)^{m+1}m!\over (2i+b)^{m+1}}+p\sum\limits_{j=p+1}^{\infty}{(-1)^{m+1}m!\over (2j+b)^{m+1}}\\
&=\frac12\sum\limits_{i=1}^p{(-1)^{m+1}m! \over (2i+b)^{m}}-{b\over 2}\sum\limits_{i=1}^p{(-1)^{m+1}m!\over (2i+b)^{m+1}}+p\sum\limits_{j=p+1}^{\infty}{(-1)^{m+1}m!\over (2j+b)^{m+1}}\\
&=-{m\over 2^{m+1}}\sum\limits_{i=1}^p{(-1)^{m}(m-1)! \over \left(i+{b\over 2}\right)^m}-{b\over 2^{m+2}}\sum\limits_{i=1}^\infty{(-1)^{m+1}m!\over \left(i+{b\over 2}\right)^{m+1}}\\
&\hspace{4cm}+\left({2p+b\over 2^{m+2}}\right)\sum\limits_{j=0}^{\infty}{(-1)^{m+1}m!\over \left(j+{b\over 2}+p+1\right)^{m+1}}.
\end{align*}
Then, using \eqref{eq:PolygammaSeriesFormula} and estimate \eqref{eq:PolygammaBound} we obtain
\begin{align*}
\left|T(b)\right|&\leq{(2m-1)(m-2)!\over 4|b+2|^{m-1}}+{m!\over |b+2|^{m}}+{|2p+b|(m-1)!\over 4|b+2p+2|^{m}}+{|2p+b|m!\over 2|b+2p+2|^{m+1}}\\
&\leq{(2m-1)(m-2)!\over 4|b+2|^{m-1}}+{m!\over |b+2|^{m}}+{(m-1)!\over 4|b+2p+2|^{m-1}}+{m!\over 2|b+2p+2|^{m}}.
\end{align*}
Finally, we get
\begin{align*}
\left|\frac14\sum\limits_{j=1}^k\psi^{(1)}\left({j+a\over 2}\right)\right|&\leq\left|T(a)\right|+\left|T(a-1)\right|+\left|{1\over 2^{m+1}}\psi^{(m)}\left({k+a\over 2}\right)\right|\\
&\leq{(2m-1)(m-2)!\over 2|a+1|^{m-1}}+{2m!\over |a+1|^{m}}+{(m-1)!\over 2|a+k|^{m-1}}+{m!\over |a+k|^{m}}\\
&\qquad\qquad+{(m-1)!\over 2|a+k|^{m}}+{m!\over |a+k|^{m+1}}\\
&\leq {5\over 2} {m!\over |a+1|^{m-1}}+{3\over 2}{m!\over |a+k|^{m-1}}.
\end{align*}
This proves the claim.
\end{proof}

\section{Cumulant estimates and their consequences}\label{sec:CumulantsResults}

This section is devoted to a description of the asymptotic probabilistic behaviour of the random variables
\begin{equation}\label{eq:DefYmun}
Y_{\mu,n}:=\log V_n(Z_{\mu}),\qquad \mu\in(-2,\infty),
\end{equation}
in the high dimensional regime, that is, as $n\to\infty$. We recall at this occasion that $Z_\mu$ is a random simplex with distribution 
$\Pr_\mu^0$ given by \eqref{eq:DefWeightedCellIndicator}. That is, $Z_\mu$ is a typical $V_n^{\mu+1}$-weighted random simplex from the Poisson-Delaunay tessellation in $\R^n$ with intensity $\gamma$. In our set-up we will allow the parameter $\mu$ and the intensity $\gamma$ of the underlying Poisson point process to depend on the dimension parameter $n$, although this dependence is suppressed in our notation for simplicity.

Our strategy is as follows. First, we derive in Section \ref{subsec:CumGen} asymptotic expansions for the expectation and the variance and provide sharp bounds for the cumulants of the random variables $Y_{\mu,n}$. They are specialized in Section \ref{subsec:CumSpecial} for different growth rates of $\mu$ with respect to $n$. These estimates are then used in Section \ref{subsec:MainResults} to prove our main result, Theorem \ref{thm:Main}, based on the general limit theorems for large deviations from \cite{SS91}. In particular, this includes part (i) and (ii) of Theorem \ref{thm:Intro} from the introduction as a special case.

\subsection{General bounds for cumulants}\label{subsec:CumGen}

Given an $m\in\mathbb{N}$ and a random variable $X$ with $\Ex|X|^m<\infty$, let $c_m(X)$ be the $m$-th order cumulant of $X$ defined by
\[
c_m(X)=(-\mathfrak{i})^m\frac{\dd^m}{\dd t^m}\log \Ex\left[e^{\mathfrak{i}tX}\right]\Big|_{t=0},
\]
where $\mathfrak{i}$ is the imaginary unit.
The aim of this section is to derive the bounds for the cumulants of random variable $Y_{\mu,n}$. In addition, we will determine an asymptotic expansion for the expectation and the variance of $Y_{\mu,n}$, as $n\to\infty$. As a special case, this provides a proof of Theorem \ref{thm:Intro} (i) from the introduction when we take $\mu=-1$.

\begin{proposition}\label{lem:CumulantBounds}
For any $\mu\in(-2,\infty)$ and $n\in\N$, $n\ge 2$ we have
\begin{align*}
\Ex Y_{\mu,n}&=-{n\over 2}\log n-n\log\sqrt{2\pi}-\log\gamma +\left({\mu\over 2}+{7\over 4}\right)\log\left(n+\mu\right)+\frac{1}{2}\log n\\
&\qquad-{\mu+1\over 2}\psi(\mu+3)-{1\over 4}\psi\left({\mu\over 2}+2\right)+O(1),\\
\Var Y_{\mu,n}&={3-n\over 2(n+\mu)}+{2n\over (n+\mu)^2}+\frac12\log(n+\mu)+{1\over 2}-\frac12\psi(\mu+3)-{\mu+2\over 2}\psi^{(1)}\left(\mu+3\right)\\
&\qquad+{1\over 8}\psi^{(1)}\left({\mu+3\over 2}\right)+O\left({1\over (n+\mu)^2}\right).
\end{align*}
Moreover, for $m\in\N$, $m\geq 3$ we have
\begin{align*}
\left|c_m(Y_{\mu,n})\right|&\leq{(3n+4)(m-2)!\over 2(n+\mu)^{m-1}}+{(2n+3)(m-1)!\over (n+\mu)^{m}}+{4(m-1)!\over (\mu+3)^{m-2}}.
\end{align*}
\end{proposition}
\begin{proof}
Substituting the moment formula \eqref{eq:Moments} into the definition of the cumulants and using the relation $\Gamma(x+1)=x\Gamma(x)$, $x\in(0,\infty)$, we see that
\begin{align*}
c_m(Y_{\mu,n})&=\frac{\dd^m}{\dd s^m}\log \Ex\left[V_n(Z_{\mu})^{s}\right]\Big|_{s=0}={\bf 1}_{\{m=1\}}\left(\log \Gamma\left({n\over 2}+1\right)-\log\gamma-\frac{n}{2}\log\pi-\log n!\right)\\
&\qquad+\frac{\dd^m}{\dd s^m}\Bigg[\log\Gamma\left(n+\mu+s\right)+\log\Gamma\left({(n+1)(n+\mu)\over 2}+{n+1 \over 2}s\right)\\
&\qquad-\log\Gamma\left({n(n+\mu+1)\over 2}+{n\over 2}s\right)-(n+1)\log\Gamma\left({n+\mu\over 2}+{s\over 2}\right)\\
&\qquad-(n-1)\log(n+\mu+s)+\sum\limits_{i=1}^n\log\Gamma\left({i+\mu+2\over 2}+{s\over 2}\right)\Bigg]\Bigg|_{s=0}
\end{align*}
after simplifications. The above equality can be rewritten in terms of polygamma functions as follows:
\begin{align*}
c_m(Y_{\mu,n})&={\bf 1}_{\{m=1\}}\left(\log\Gamma\left({n\over 2}+1\right)-\log\gamma-\frac{n}{2}\log\pi-\log n!\right)+\psi^{(m-1)}\left(n+\mu\right)\\
&\qquad+{(n+1)^m\over 2^m}\psi^{(m-1)}\left({(n+1)(n+\mu)\over 2}\right)-{n^m\over 2^m}\psi^{(m-1)}\left({n(n+\mu+1)\over 2}\right)\\
&\qquad-{n+1\over 2^m}\psi^{(m-1)}\left({n+\mu\over 2}\right)+{1\over 2^m}\sum\limits_{i=1}^n\psi^{(m-1)}\left({i+\mu+2\over 2}\right)-{n-1\over (n+\mu)^m}.
\end{align*}
We now distinguish the cases $m=1$, $m=2$ and $m\geq 3$. If $m=1$, $c_1(Y_{\mu,n})=\Ex Y_{\mu,n}$ and the above expression is
\begin{align*}
\Ex Y_{\mu,n}&=\log \Gamma\left(\frac{n}{2}\right)+\log n-\log\gamma-\frac{n}{2}\log\pi-\log n!+\psi\left(n+\mu\right)-{n+1\over 2}\psi\left({n+\mu\over 2}\right)\\
&+{n+1\over 2}\psi\left({(n+1)(n+\mu)\over 2}\right)-\frac{n}{2}\psi\left({n(n+\mu+1)\over 2}\right)+\frac12\sum\limits_{i=1}^{n}\psi\left({i+\mu+2\over 2}\right)+O(1).
\end{align*}
Using the asymptotic relations \eqref{eq:AsympGamma} and \eqref{eq:Stirling} together with Proposition \ref{prop:DigammaSum} and the identity $\psi(x+1)=\psi(x)+{1\over x}$, $x\in(0,\infty)$, we obtain
\begin{align*}
\Ex Y_{\mu,n}&=-{n\over 2}\log n-{n\over 2}\log 2+{n\over 2}-{n\over 2}\log\pi-\log\gamma+\psi\left(n+\mu\right)+\frac{n+1}{2}\psi\left({(n+1)(n+\mu)\over 2}\right)\\
&\qquad-\frac{n}{2}\psi\left({n(n+\mu+1)\over 2}\right)-{n+1\over 2}\psi\left({n+\mu\over 2}\right)+ {n+\mu+1\over 2}\psi (n+\mu)+{1\over 4}\psi\left({\mu+n\over 2}\right)\\
&\qquad-{\mu+1\over 2}\psi(\mu+3)-{1\over 4}\psi\left({\mu\over 2}+2\right)-{n\over 2}\left(1+\log 2\right) +O(1).
\end{align*}
Applying now the asymptotics \eqref{eq:AsympDigamma} for the digamma function we conclude that
\begin{align*}
&\Ex Y_{\mu,n}=-{n\over 2}\log n-n\log\sqrt{2\pi}-\log\gamma +\left({n+\mu\over 2}+{7\over 4}\right)\log\left(n+\mu\right)\\
&\quad+\frac{n+1}{2}\log(n+1)-\frac{n}{2}(\log n+\log(n+\mu+1))-{\mu+1\over 2}\psi(\mu+3)-{1\over 4}\psi\left({\mu\over 2}+2\right)+O(1).
\end{align*}
Finally, taking into account that $\log (x+1)=\log x+O(1/x)$, $x\in(0,\infty)$, we obtain
\begin{align*}
\Ex Y_{\mu,n}&=-{n\over 2}\log n-n\log\sqrt{2\pi}-\log\gamma +\left({\mu\over 2}+{7\over 4}\right)\log\left(n+\mu\right)+\frac{1}{2}\log n\\
&\qquad-{\mu+1\over 2}\psi(\mu+3)-{1\over 4}\psi\left({\mu\over 2}+2\right)+O(1).
\end{align*}
Next, we turn to the case $m=2$. Since $c_2(Y_{\mu, n})=\Var Y_{\mu,n}$ we have
\begin{align*}
\Var Y_{\mu,n}&=\psi^{(1)}\left(n+\mu\right)+{(n+1)^2\over 4}\psi^{(1)}\left({(n+1)(n+\mu)\over 2}\right)-{n^2\over 4}\psi^{(1)}\left({n(n+\mu+1)\over 2}\right)\\
&\qquad-{n+1\over 4}\psi^{(1)}\left({n+\mu\over 2}\right)+{1\over 4}\sum\limits_{i=1}^n\psi^{(1)}\left({i+\mu+2\over 2}\right)-{n-1\over (n+\mu)^2}.
\end{align*}
Using now Proposition \ref{prop:DerivativeDigammaSum} together with \eqref{eq:AsympDigamma} and  \eqref{eq:AsympPolygamma} we conclude that
\begin{align*}
\Var Y_{\mu,n}&={1\over n+\mu}-{n\over 2(n+\mu+1)}+{3n\over 2(n+\mu)^2}+\frac12\log(n+\mu+3-c)+{n+\mu+1\over 2(n+\mu+3-c)}\\
&\qquad-\frac12\psi(\mu+3)-{\mu+2\over 2}\psi^{(1)}\left(\mu+3\right)+{1\over 8}\psi^{(1)}\left({\mu+3\over 2}\right)+O\left({1\over (n+\mu)^2}\right).
\end{align*}
Using a Taylor expansion of the functions $\log x$ and $1/x$ we see that
\begin{align*}
\log(n+\mu+3-c)&=\log (n+\mu)+{3-c\over n+\mu}+O\left({1\over (n+\mu)^2}\right),\\
{n\over 2(n+\mu+1)}&={n\over 2(n+\mu)}-{n\over 2(n+\mu)^2}+O\left({1\over (n+\mu)^2}\right),
\end{align*}
and, thus,
\begin{align*}
\Var Y_{\mu,n}&={3-n\over 2(n+\mu)}+{2n\over (n+\mu)^2}+\frac12\log(n+\mu)+{1\over 2}-\frac12\psi(\mu+3)\\
&\qquad-{\mu+2\over 2}\psi^{(1)}\left(\mu+3\right)+{1\over 8}\psi^{(1)}\left({\mu+3\over 2}\right)+O\left({1\over (n+\mu)^2}\right).
\end{align*}
This proves the first two assertions of the proposition. 

We turn now to the case that $m\geq 3$. Applying the estimate \eqref{eq:PolygammaBound} and Proposition \ref{prop:DerivativeDigammaSumBound} for $n\ge 2$, we get
\begin{align*}
\left|c_m(Y_{\mu,n})\right|&\leq\left|\psi^{(m-1)}\left(n+\mu\right)\right|+{(n+1)^m\over 2^m}\left|\psi^{(m-1)}\left({(n+1)(n+\mu)\over 2}\right)\right|\\
&\qquad+{n^m\over 2^m}\left|\psi^{(m-1)}\left({n(n+\mu+1)\over 2}\right)\right|+{n+1\over 2^m}\left|\psi^{(m-1)}\left({n+\mu\over 2}\right)\right|\\
&\qquad+\left|{1\over 2^m}\sum\limits_{i=1}^n\psi^{(m-1)}\left({i+\mu+2\over 2}\right)\right|+{n-1\over (n+\mu)^m}\\
&\leq{(m-2)!\over (n+\mu)^{m-1}}+{(n+4)(m-1)!\over (n+\mu)^{m}}+{(n+1)(m-2)!\over 2(n+\mu)^{m-1}}+{n(m-2)!\over 2(n+\mu)^{m-1}}\\
&\qquad+{(n+1)(m-2)!\over 2(n+\mu)^{m-1}}+{3(m-1)!\over (\mu+3)^{m-2}}+{n-1\over (n+\mu)^m}\\
&\leq{(3n+4)(m-2)!\over 2(n+\mu)^{m-1}}+{(2n+3)(m-1)!\over (n+\mu)^{m}}+{3(m-1)!\over (\mu+3)^{m-2}}.
\end{align*}
This completes the proof.
\end{proof}

\subsection{Asymptotics and bounds in different regimes}\label{subsec:CumSpecial}

In this section we simplify the asymptotic expansions for $\Ex Y_{\mu,n}$ and $\Var Y_{\mu,n}$ as well as the cumulant estimates provided in Proposition \ref{lem:CumulantBounds} for a number of special choices for the growth rate of $\mu$ with respect to the dimension parameter $n$. We recall that the symbol $\gamma$ stands for the intensity of the underlying Poisson point process.

\subsubsection{The case when $n\rightarrow \infty$ and $\mu$ is fixed}

From Proposition \ref{lem:CumulantBounds} we get
\begin{align*}
\Ex Y_{\mu,n}&=-{n\over 2}\log n-n\log\sqrt{2\pi}-\log\gamma +\left({\mu\over 2}+{9\over 4}\right)\log n+O(1)= -{n\over 2}\log n-\log\gamma+O(n)
\end{align*}
and
\begin{align*}
\Var Y_{\mu,n}&=\frac12\log(\mu+n+3)+O(1)= \frac12\log n+O(1).
\end{align*}
In addition, for all $m\geq 3$ we have the bound
\[
\left|c_m(Y_{\mu,n})\right|\leq{6(m-1)!\over (\mu+3)^{m-2}}.
\]
for $n\ge 3$.

\subsubsection{The case when $n\rightarrow \infty$ and $\mu=n^{\alpha}$, $\alpha \in (0,1)$}

By Proposition \ref{lem:CumulantBounds}, Equality \eqref{eq:AsympDigamma} and a Taylor expansion of logarithm we obtain
\begin{align*}
\Ex Y_{\mu,n}&=-{n\over 2}\log n+{1-\alpha\over 2}n^{\alpha}\log n-n\log\sqrt{2\pi}-\log\gamma\\
&\hspace{3cm} +{7-3\alpha\over 4}\log n+\frac12\sum\limits_{i=1}^{\infty}{(-1)^{i+1}\over i!n^{i-(i+1)\alpha}}+O(1)\\
&= -{n\over 2}\log n-\log\gamma + O(n).
\end{align*}
Analogously, using \eqref{eq:AsympDigamma}, \eqref{eq:AsympPolygamma} and a Taylor expansion of the logarithm and the function $1/x$ we get
\begin{align*}
\Var Y_{\mu,n}&={3-n\over 2(n+n^\alpha)}+{2n\over (n+n^\alpha)^2}+{1\over 2}\log(n+n^\alpha)+{1\over 2}-{1\over 2}\log(n^\alpha+3)+{1\over 4(n^\alpha+3)}\\
&\qquad\qquad-{n^\alpha+2\over 2(n^\alpha+3)}-{n^\alpha+2\over 4(n^\alpha+3)^2}+{1\over 4(n^\alpha+3)}+O\left({1\over n^{2\alpha}}\right)\\
&={1-\alpha\over 2}\log n+ O(1),
\end{align*}
and
\begin{align*}
\left|c_m(Y_{\mu,n})\right|&\leq {6(m-1)!\over n^{\alpha(m-2)}}.
\end{align*}
for all $m\geq 3$ and $n\geq 3$.

\subsubsection{The case when $n\rightarrow \infty$ and $\mu=\alpha n$ for some fixed $\alpha$}

Substituting $\mu=\alpha n$ with $\alpha\in(0,\infty)$ into Proposition \ref{lem:CumulantBounds} and using the asymptotic relation \eqref{eq:AsympDigamma} we get
\begin{align*}
\Ex Y_{\mu,n}&=-{n\over 2}\log n-n\log\sqrt{2\pi}-\log\gamma +\left({\alpha n\over 2}+{7\over 4}\right)(\log n + \log\left(1+\alpha\right))+\frac{1}{2}\log n\\
&\qquad\qquad-{\alpha n+1\over 2}(\log\alpha +\log n)-{1\over 4}\log n+O(1)\\
&=-{n\over 2}\log n-{n\over 2}(\log 2\pi-\alpha\log(1+\alpha^{-1}))-\log\gamma +\log n+O(1)\\
&= -{n\over 2}\log n-\log\gamma + O(n).
\end{align*}
Analogously, using \eqref{eq:AsympPolygamma} we obtain
\begin{align*}
\Var Y_{\mu,n}&=-{1\over 2(1+\alpha)}+\frac12\log((1+\alpha)n)-\frac12\log(\alpha n)+O\left({1\over n}\right)\\
&= \frac12\log\left(1+{1\over \alpha}\right)-{1\over 2(1+\alpha)}+O\left({1\over n}\right).
\end{align*}
Finally we have that
\begin{align*}
\left|c_m(Y_{\mu,n})\right|&\leq {6(m-1)!\over(\alpha n)^{m-2}}
\end{align*}
for all $m\geq 3$ and $n\geq 3$.

\subsubsection{The case when $n\rightarrow \infty$ and $n-\mu = o(n)$}

By Proposition \ref{lem:CumulantBounds} and Identity \eqref{eq:AsympDigamma} applied to this case we obtain
\begin{align*}
\Ex Y_{\mu,n}&=-{n\over 2}\log n-n\log\sqrt{2\pi}-\log\gamma +\left({n-o(n)\over 2}+{7\over 4}\right)(\log n+\log 2)+\frac{1}{2}\log n\\
&\qquad\qquad-{n-o(n)+1\over 2}\log n+O(1)\\
&= -{n\over 2}\log n-\log\gamma - O(n).
\end{align*}
Similarly, from \eqref{eq:AsympDigamma} and \eqref{eq:AsympPolygamma} we get
\begin{align*}
\Var Y_{\mu,n}&=-{1\over 4}+\frac12\log\left({2n-o(n)+3 \over n-o(n)+3 }\right)+O\left({1\over n}\right)={1\over 2}\log 2-{1\over 4}+O\left({1\over n}\right),
\end{align*}
and also 
\begin{align*}
\left|c_m(Y_{\mu,n})\right|&\leq {6(m-1)!\over n^{m-2}}
\end{align*}
for all $m\geq 3$ and $n\geq 3$.

\subsubsection{The case when $n$ is fixed and $\mu\rightarrow \infty$}

From Proposition \ref{lem:CumulantBounds} and \eqref{eq:AsympDigamma} we conclude that
\begin{align*}
\Ex Y_{\mu,n}&=-\log\gamma+\left({\mu\over 2}+{7\over 4}\right)\log \mu-{\mu+1\over 2}\log \mu-{1\over 4}\log \mu+O(1)
\\&= \log \mu-\log\gamma+O(1).
\end{align*}
Similarly, using \eqref{eq:AsympDigamma}, \eqref{eq:AsympPolygamma} and a Taylor expansion of the logarithm we get
\begin{align*}
\Var Y_{\mu,n}&={3-n\over 2(n+\mu)}+\frac12\log(n+\mu)-\frac12\log(\mu+3)+{3\over 4\mu}+O\left({1\over \mu^2}\right)\\
&={3\over 4\mu}+O\left({1\over \mu^2}\right),
\end{align*}
and also
\begin{align*}
\left|c_m(Y_{\mu,n})\right|&\leq {7n+16\over 8}\,{(m-1)!\over \mu^{m-2}}
\end{align*}
for all $m\geq 3$ and $n\geq 3$.

\subsubsection{The case when $\mu\rightarrow \infty$ and $n =\mu^{\alpha}$, $\alpha \in (0,1)$}

Applying Proposition \ref{lem:CumulantBounds}, equality \eqref{eq:AsympDigamma} and a Taylor expansion of the logarithm we obtain
\begin{align*}
\Ex Y_{\mu,n}&=-{\alpha\over 2}\mu^{\alpha}\log \mu-\mu^{\alpha}\log\sqrt{2\pi}-\log\gamma +\left({\mu\over 2}+{7\over 4}\right)\log(\mu^{\alpha}+\mu)\\
&\qquad\qquad+{\alpha\over 2}\log\mu-{\mu+1\over 2}\log(\mu+3)-{1\over4}\log\left({\mu\over 2}+2\right)+O(1)\\
&= -{\alpha\over 2}\mu^{\alpha}\log \mu-{\mu^{\alpha}\over 2}(\log\pi-1+\log 2)-\log\gamma +{1+\alpha\over 2}\log\mu+{\mu\over 2}\sum\limits_{i=2}^{\infty}{(-1)^{i+1}\over i!\mu^{i(1-\alpha)}}\\
&=-{\alpha\over 2}\mu^{\alpha}\log \mu-\log\gamma+O(\mu^{\alpha}).
\end{align*}
Analogously, using \eqref{eq:AsympDigamma}, \eqref{eq:AsympPolygamma} and a Taylor expansion of the logarithm as well as of the function $1/x$ we get
\begin{align*}
\Var Y_{\mu,n}&={3-\mu^{\alpha}\over 2(\mu+\mu^\alpha)}+{2\mu^{\alpha}\over (\mu+\mu^\alpha)^2}+{1\over 2}\log(\mu+\mu^\alpha)+{1\over 2}-{1\over 2}\log(\mu+3)+{1\over 4(\mu+3)}\\
&\qquad\qquad-{\mu+2\over 2(\mu+3)}-{\mu+2\over 4(\mu+3)^2}+{1\over 4(\mu+3)}+O\left({1\over (\mu+ \mu^{\alpha})^2}\right)\\
&={3\over 4\mu}+{1\over 2}\sum\limits_{i=2}^{\infty}\left(1-{1\over i!}\right){(-1)^i\over \mu^{i(1-\alpha)}}+ O\left(1\over \mu^{2-\alpha}\right).
\end{align*}
Thus, for $\alpha<{1\over 2}$ we have
\begin{align*}
\Var Y_{\mu,n}&={3\over 4\mu}+ O\left(1\over \mu^{2-2\alpha}\right),
\end{align*}
for $\alpha={1\over 2}$ we get
\begin{align*}
\Var Y_{\mu,n}&={1\over \mu}+ O\left(1\over \mu^{3/2}\right),
\end{align*}
and for $\alpha>{1\over 2}$ we obtain
\begin{align*}
\Var Y_{\mu,n}&={1\over 4\mu^{2(1-\alpha)}}+ O\left({1\over \mu}+{1\over \mu^{3(1-\alpha)}}\right).
\end{align*}
Finally, we conclude that
\begin{align*}
\left|c_m(Y_{\mu,n})\right|&\leq {6(m-1)!\over \mu^{m-2}}.
\end{align*}
for all $m\geq 3$ and $n\geq 3$.
%\textcolor{red}{Shall we skip this case? What about $\mu>>d$?}

\subsection{Limit theorems for weighted random simplices}\label{subsec:MainResults}

Having derived cumulant bounds for the random variables $Y_{\mu, n}$ we are now in the position to prove a number of probabilistic limit theorems for the logarithmic volume of the weighted random simplices $Z_\mu$ in a Poisson-Delaunay tessellation. For this we use the following lemma. It summarizes results from \cite{DE13b} and \cite{SS91} in a simplified form (see also \cite{ERS,GT18a,GT18b}). As above, we shall write $c_m(X)$, $m\in\mathbb{N}$, for the $m$-th cumulant of a random variable $X$ with $\Ex|X|^m<\infty$. Moreover, by $\Phi(\,\cdot\,)$ we denote the distribution function of a standard Gaussian random variable. 

In what follows we will obtain a concentration inequality (see Theorem \ref{thm:Main} (i)), Berry-Esseen bounds (see Theorem \ref{thm:Main} (iv)) and investigate moderate deviations (see Theorem \ref{thm:Main} (ii) and (iii)). Let us recall here for convenience the definition of moderate deviations principle for a sequence of random variables. Given a sequence $(\nu_n)_{n\in\mathbb{N}}$ of probability measures on a topological space $E$, we say that it fulfils a large deviations principle with speed $a_n$ and (good) rate function $I:E\rightarrow [0;\infty]$, if $I$ is lower semi-continuous and has compact level sets, and if for every Borel set $B\subseteq E$ we have
\[
-\inf_{x\in \inter(B)} I(x)\leq \liminf_{n\rightarrow \infty}a_n^{-1}\log\nu_n(B)\leq \limsup_{n\rightarrow\infty}a_n^{-1}\log \nu_n(B)\leq -\inf_{x\in \cl(B)} I(x),
\]
where $\inter(B)$ and $\cl(B)$ stand for the interior and the closure of $B$, respectively. We say that a sequence $(X_n)_{n\in\mathbb{N}}$ of random elements in $E$ satisfies a large deviations principle if the sequence of their distributions does. Moreover, if the rescaling $a_n$ lies between that of a law of large numbers and that of a distributional (often a central) limit theorem, we will say that a sequence $(X_n)_{n\in\mathbb{N}}$ satisfies a moderate deviations principle with speed $a_n$ and rate function $I$, cf.\ \cite{DZ}.

 \begin{lemma}\label{lem:ConsequenceCumBounds}
 	Let $(X_n)_{n\in\N}$ be a sequence of random variables with $\Ex[X_n] = 0$ and $\Var[X_n] = 1$ for all $n\in\N$. Suppose that, for all $m\in \N$, $m\ge 3$ and sufficiently large $n$,
 	\begin{align}\label{eq:CumBoundAbstract}
 	|c_m(X_n)| \leq \frac{(m!)^{1 + \delta}}{(\Delta_n)^{m-2}}
 	\end{align}
 	with a constant $\delta \in [0,\infty)$ not depending on $n$ and constants $\Delta_n \in (0,\infty)$ that may depend on $n$.  Then the following assertions are true.
 	\begin{itemize}
 		\item[(i)] For all $y\in[0,\infty)$ and sufficiently large $n$,
 		\begin{align*}
 		\Pr(|X_n| \geq y) \leq 2 \exp\left(-{y^2\over 2\big[1+(y/\Delta_n^{1/(1+2\delta)})^{1+2\delta\over 1+\delta}\big]}\right). 
 		\end{align*}
 		\item[(ii)] There exist constants $c_1,c_2 \in (0,\infty)$ only depending on $\delta$ such that for sufficiently large $n$ and $0\leq y\leq c_1\,(\Delta_n)^{\frac{1}{1 + 2\delta}}$,
 		\begin{align*}
 		\left|\log \frac{\Pr(X_n \geq y)}{1 - \Phi(y)}\right| &\leq c_2\, (1 + y^3)\, (\Delta_n)^{-\frac{1}{(1+2\delta)}} \quad and\\
 		\left|\log \frac{\Pr(X_n \leq -y)}{\Phi(-y)}\right| &\leq c_2\, (1 + y^3)\,(\Delta_n)^{-\frac{1}{(1+2\delta)}}.
 		\end{align*}
 		\item[(iii)] Let $(a_n)_{n \in\N}$ be a sequence of positive real numbers such that
 		\begin{align*}
 		\lim\limits_{n \rightarrow \infty} a_n = \infty \quad \text{and}\quad \lim\limits_{n \rightarrow \infty} a_n\, \Delta_n^{-\frac{1}{1 + 2\delta}} = 0.
 		\end{align*}
 		Then $(a_n^{-1} X_n)_{n\in\N}$ satisfies a moderate deviations principle on $\mathbb{R}$ with speed  $a_n^2$ and rate function $I(x) = \frac{x^2}{2}$.
 		\item[(iv)] One has the Berry-Esseen bound
 		\begin{align*}
 		\sup\limits_{y\in \mathbb{R}} |\Pr(X_n \leq y) - \Phi(y)| \leq c\, (\Delta_n)^{-1/(1+2\delta)}
 		\end{align*}
 		with a constant $c\in (0,\infty)$ that only depends on $\delta$.
 	\end{itemize} 
 \end{lemma} 
 
 \begin{remark}\rm 
Lemma \ref{lem:ConsequenceCumBounds} (ii) is a simplified form of a Cram\'er-Petrov-type asymptotic expansion that follows from the cumulant bound \eqref{eq:CumBoundAbstract} together with \cite[Lemma 2.2 and 2.3]{SS91}. We leave the formulation of the more precise (and much more involved) statement that follows from these results to the reader.
 \end{remark}

We shall distinguish the following three regimes:
\begin{itemize}
\item[(R1)] $n\to\infty$ and $\mu\in(-2,\infty)$ fixed,
\item[(R2)] $n\to\infty$ and $\mu=n^\alpha$ for some $\alpha\in(0,1)$,
\item[(R3)] $n\to\infty$ and $n-\mu=o(n)$,
\item[(R4)] $n\to\infty$ and $\mu=\alpha n$ for some $\alpha\in(0,\infty)$.
\end{itemize}
Moreover, for $n\in\N$ we define the quantity
\begin{align*}
\epsilon_n := \begin{cases}
\sqrt{\log n} &: \mu\text{ fixed},\\
n^\alpha\sqrt{\log n} &: \mu=n^\alpha,\alpha\in(0,1),\\
n &: \mu=\alpha n,\alpha\in(0,\infty)\text{ or }n-\mu=o(n),
\end{cases}
\end{align*}
as well as the centred and normalized random random variables
	\[
	\tilde Y_{\mu, n}:=\frac{Y_{\mu, n}-\Ex Y_{\mu, n}}{\sqrt{\Var Y_{\mu, n}}}.
	\]
We are now ready to present our main result. Note that part (ii) of Theorem \ref{thm:Intro} from the introduction corresponds to the special choice $\mu=-1$. 

\begin{theorem}\label{thm:Main}
Suppose that $n$ and $\mu$ are such that we are in one of the regimes (R1), (R2), (R3) or (R4). Then the following assertions are true.
\begin{itemize}
\item[(i)] There exists a constant $c\in(0,\infty)$ such that for all $y\in[0,\infty)$ and large enough $n$ one has that
$$
\Pr(\widetilde{Y}_{\mu,n}\geq y) \leq 2\,\exp\Big(-{y^2\over 2+cy\epsilon_n^{-1}}\Big).
$$

\item[(ii)] There are constants $c_1,c_2\in(0,\infty)$ such that for large enough $n$ and $0\leq y\leq c_1\epsilon_n$ one has that
\begin{align*}
\Big|\log{\Pr(\widetilde{Y}_{\mu,n}\geq y)\over 1-\Phi(y)}\Big| \leq c_2\,(1+y^3)\,\epsilon_n^{-1},\\
\Big|\log{\Pr(\widetilde{Y}_{\mu,n}\leq -y)\over \Phi(-y)}\Big| \leq c_2\,(1+y^3)\,\epsilon_n^{-1}.
\end{align*}

\item[(iii)] Let $(a_n)_{n\in\N}$ be a sequence of positive real numbers such that
$$
\lim_{n\to\infty}a_n=\infty\qquad\text{and}\qquad \lim_{n\to\infty}a_n\,\epsilon_n^{-1} = 0.
$$
Then the sequence of random variables $(a_n^{-1}\widetilde{Y}_{\mu,n})_{n\in\N}$ satisfies a moderate deviations principle on $\R$ with speed $a_n^2$ and rate function $I(x)=x^2/2$.

\item[(iv)] There exists a constant $c\in(0,\infty)$ such that for all $n\geq 3$,
$$
\sup_{y\in\R}|\Pr(\widetilde{Y}_{\mu,n}\leq y)-\Phi(y)| \leq c\,\epsilon_n^{-1}.
$$
\end{itemize}
\end{theorem}
\begin{proof}
	From the estimates presented in the previous section we obtain
	\[
	\left|c_m\left(\tilde Y_{\mu, n}\right)\right|=\frac{\left|c_m\left(Y_{\mu, n}\right)\right|}{\left(\Var Y_{\mu, n}\right)^{m/2}}\leq A^m{m!\over\epsilon_n^{m-2}}\leq\bigg({A\max\{1,A^2\}\over\epsilon_n}\bigg)^{m-2}m!,
	\]
	for all $n\geq 3$. Here, $A\in(0,\infty)$ is an absolute constant, except for the cases where $\mu=n^\alpha$ or $\mu=\alpha n$, where it depends additionally on the choice of $\alpha$.
	Thus $\tilde Y_{\mu, n}$ satisfies condition \eqref{eq:CumBoundAbstract} with $\Delta_n=\epsilon_n/(A\max\{1,A^2\})$ and $\delta=0$. This completes the proof.
\end{proof}

\begin{remark}\rm 
The proof of Theorem \ref{thm:Main} shows that the constants $c$ in part (i) and (iv) and the constants $c_1,c_2$ in part (ii) are in fact absolute constants in regimes (R1) and (R3) and that they only depend on the additional parameter $\alpha$ in regimes (R2) and (R4).
\end{remark}

\section{Mod-$\phi$ convergence and large deviations principle}\label{sec:ModPhi}

In this final section we investigate mod-$\phi$ convergence and the large deviations behaviour of the logarithmic volume of the random simplex $Z_{\mu}$. The notion of mod-$\phi$ convergence has been recently introduced and studied in \cite{DKN15, JKN11}. It is a powerful tool which leads to a whole collection of limit theorems including an extended version of the central limit theorem, a local limit theorem, precise moderate and large deviations and Cram\'er-Petrov type asymptotic expansions. For more references and a survey of the topic we refer the reader to \cite{ModPhiBook}. We remark that \textit{some} of the results established in the previous section for fixed $\mu$ will also follow once we have established mod-$\phi$ convergence.

The main idea behind the concept of mod-$\phi$ convergence of a sequence of random variables is to look for a suitable renormalization of the moment generating functions (considered on the complex plane $\mathbb{C}$) of these random variables. There are a several versions of the definition of mod-$\phi$ convergence, we will consider the one from \cite[Definition 1.1]{ModPhiBook}. Let $(X_n)_{n\in\mathbb{N}}$ be a sequence of real-valued random variables, and let us denote by $\varphi_n(z)=\Ex[e^{zX_n}]$ their moment generating functions, which are assumed to exist in a strip 
\[
S_{(a,b)}:=\left\{z\in\C\colon a<\Re z<b\right\},
\]
where $a<0<b$ are extended real numbers. Assume there exists a non-constant infinitely divisible distribution $\phi$ with moment generating function $\int_{\R}e^{zx}\phi(\dd x)=\exp(\eta(z))$, which is well defined on $S_{(a,b)}$, and an analytic function $\psi(z)$ which does not vanish on the real part of $S_{(a,b)}$, such that 
\[
\exp\left(-w_n\eta(z)\right)\varphi_n(z)\rightarrow \psi(z),\qquad n\to\infty,
\]
locally uniformly in $z\in S_{(a,b)}$ for some sequence $(w_n)_{n\in\mathbb{N}}$ converging to infinity. Then we say that the sequence $(X_n)_{n\in\mathbb{N}}$ converges in the mod-$\phi$ sense on $S_{(a,b)}$ with parameters $(w_n)_{n\in\mathbb{N}}$ and limiting function $\psi$. Especially, if $\eta(z)=z^2/2$ is the Gaussian exponent, one speaks about mod-Gaussian convergence.

Mod-$\phi$ convergence for the log-volume of random simplices generated by random vectors distributed according to the Gaussian distribution, the beta distribution, the beta' distribution or the uniform distribution on the sphere was recently studied in \cite{EKGammaMoments,GKT17}. We remark that although \cite{EKGammaMoments} studies very general models with so-called gamma type moments, our random variables do not precisely fit into this framework.

\subsection{The Barnes $G$-function}

Our result about the mod-$\phi$ convergence of the random variables $Y_{\mu,n}$ defined in \eqref{eq:DefYmun} involves the so-called Barnes $G$-function. The Barnes $G$-function is an entire function of one complex argument $z$, which can be defined as a solution of functional equation 
\[
G(z+1)=\Gamma(z)G(z).
\]
Using induction one can deduce that for any integer $n$ one has that 
\begin{equation}\label{eq:BarnesFunct}
\prod\limits_{k=1}^n\Gamma(k+z)={G(z+n+1)\over G(z+1)},\qquad z\neq -1,-2,\ldots.
\end{equation}
In what follows we will need the following lemma, which was proved in \cite[Lemma 4.1]{GKT17}.

\begin{lemma}\label{lm:BarnesFuncAsymp}
Let $|z|\rightarrow \infty$ and let $a=a(z)\in\mathbb{C}$ be such that ${a\over z}\rightarrow \infty$. Let also $|\arg z|, |\arg (z+a)|<\pi-\epsilon$ for some $\epsilon >0$. Then
\[
\log G(z+a+1)-\log G(z+1)=a\left(z\log z-z+\log\sqrt{2\pi}\right)+{1\over 2}a^2\log z+O\left({|a|^3+1\over z}\right).
\]
\end{lemma}

\subsection{Mod-$\phi$ convergence of $Y_{\mu,n}$ for fixed $\mu$}

As before we will consider the random variables
$$
Y_{\mu,n}:=\log V_n(Z_{\mu}),\qquad \mu\in(-2,\infty),
$$
and investigate the mod-$\phi$ convergence for the sequence $(Y_{\mu,n})_{n\in\mathbb{N}}$, where $\mu$ is assumed to be fixed. It should be pointed out that other regimes for $\mu$ could also be considered using the same approach. However, we decided to restrict ourselves to the case of fixed $\mu$ since this includes the most prominent examples, namely the typical and the typical volume-weighted Delaunay simplex.

\begin{theorem}\label{thm:ModPhi}
Let $\mu\in(-2,\infty)$ be fixed and define
\[
m_n:=\log\left( {4\Gamma\left(\frac{n}{2}\right)\over \gamma\,n!\pi^{{n-1\over 2}}}\right)+\left({\mu\over 2}+{13\over 4}\right)\log\left({n\over 2}\right)-{\mu+n+1\over 2}.
\] 
Then, as $n\to\infty$, the sequence of random variables $(Y_{\mu,n}-m_n)_{n\in\N}$ converges in the mod-Gaussian sense on $S_{(-\mu-3,\infty)}$ (meaning that $\eta(z)={1\over 2}z^2$) with parameters $w_n={1\over 2}\log\left({n\over 2}\right)$ and limiting function 
\begin{equation}\label{eq:DefPsi}
\psi(z)={G\left({3+\mu\over 2}\right)G\left(2+{\mu\over 2}\right)\over G\left({3+\mu+z\over 2}\right) G\left(2+{\mu+z\over 2}\right)}.
\end{equation}
%The convergence is uniform on any compact subset of $\mathbb{C}\setminus \{-\mu-3,-\mu-4,\ldots\}$.
\end{theorem}

\begin{remark}\rm 
The centring terms $m_n$ and the limiting function $\psi$ can be simplified further for $\mu=-1$ (corresponding to the typical Delaunay simplex) and $\mu=0$ (corresponding to the typical volume-weighted Delaunay simplex) using that $G(1)=G(2)=1$ and that $G(3/2)=A^{-3/2}\pi^{1/2}e^{1/8}2^{1/24}$, where $A\approx 1.2824\ldots$ is the Glaisher-Kinkelin constant.
\end{remark}

\begin{proof}[Proof of Theorem \ref{thm:ModPhi}]
Consider the moment generating function of the random variable $Y_{\mu,n}$. From the moment formula  \eqref{eq:Moments} we have
\begin{equation}\label{eq:MomentGenartingFct}
\log \Ex \left[e^{zY_{\mu,n}}\right] =z\log\left(\frac{d\Gamma\left(\frac{n}{2}\right)}{2\gamma\pi^{n/2}n!}\right)+S_n(z)+T_n(z),
\end{equation}
where
\begin{align*}
S_n(z):&=\log \prod\limits_{i=1}^n{\Gamma\left({i+\mu\over 2}+1+{z\over 2}\right)\over \Gamma\left({i+\mu\over 2}+1\right)},\\
T_n(z):&=\left(\log\Gamma(n+\mu+1+z)-\log\Gamma(n+\mu+1)\right)\\
&\quad-\left(\log\Gamma\left({n(n+\mu+1)\over 2}+n{z\over 2}\right)-\log\Gamma\left({n(n+\mu+1)\over 2}\right)\right)\\
&\quad-(n+1)\left(\log \Gamma\left({n+\mu\over 2}+1+{z\over 2}\right) -\log\Gamma\left({n+\mu\over 2}+1\right)\right)\\
&\quad+\left(\log\Gamma\left({(n+1)(n+\mu)\over 2}+1+(n+1){z\over 2}\right)-\log \Gamma\left({(n+1)(n+\mu)\over 2}+1\right)\right).
\end{align*}
Let us start by analysing the asymptotic behaviour of $S_n(z)$. Defining $c:= n\mod 2$ and using \eqref{eq:BarnesFunct} we can rewrite $S_n(z)$ in terms of the Barnes $G$-function:
\begin{align*}
S_n(z)&=\log {G\left({3+c+\mu+n\over 2}+{z\over 2}\right)G\left({3+\mu\over 2}\right)G\left(2+{\mu+n-c\over 2}+{z\over 2}\right)G\left(2+{\mu\over 2}\right)\over G\left({3+c+\mu+n\over 2}\right)G\left({3+\mu\over 2}+{z\over 2}\right)G\left(2+{\mu+n-c\over 2}\right)G\left(2+{\mu\over 2}+{z\over 2}\right)}\\
&=\log \psi(z)+ \log G\left({3+c+\mu+n\over 2}+{z\over 2}\right) - \log G\left({3+c+\mu+n\over 2}\right)\\
&\qquad\qquad+\log G\left(2+{\mu+n-c\over 2}+{z\over 2}\right)-\log G\left(2+{\mu+n-c\over 2}\right),
\end{align*}
where $\psi$ is given by \eqref{eq:DefPsi}.
Applying now Lemma \ref{lm:BarnesFuncAsymp} and a Taylor expansion of the logarithm we conclude that
\begin{align*}
S_n(z)&=\log \psi(z)+ {z\over 2}\left({1+c+\mu+n\over 2}\log\left( {n\over 2}\right)-{n\over 2}+\log\sqrt{2\pi}\right)+{z^2\over 8}\log \left( {n\over 2}\right)\\
& \qquad+{z\over 2}\left({2-c+\mu+n\over 2}\log\left( {n\over 2}\right)-{n\over 2}+\log\sqrt{2\pi}\right)+{z^2\over 8}\log\left( {n\over 2}\right)+O\left({|z|^3+1\over n}\right)\\
&=\log \psi(z)+ {z\over 2}\left(\left({3\over 2}+\mu+n\right)\log\left( {n\over 2}\right)-n+\log(2\pi)\right)\\
&\qquad+{z^2\over 4}\log\left( {n\over 2}\right)+O\left({|z|^3+1\over n}\right).
\end{align*}
In order to compute $T_n(z)$ we will use the classical first Binet's formula \cite[page 243]{WW} for the logarithm of the gamma function, saying that
\begin{equation*}
\log \Gamma(z)=\left(z-{1\over 2}\right)\log z-z+{1\over 2}\log(2\pi)+\int\limits_{0}^{\infty}{e^{-tz}\over t}\left({1\over 2}-{1\over t}+{1\over e^t-1}\right)\dd t,\quad \Re z\in(0,\infty),
\end{equation*}
together with the relation $\Gamma(1+z)=z\Gamma(z)$. This leads to
\begin{align*}
T_n(z)&=(n+\mu+z-{1\over 2})\log (n+\mu+z)-z-(n+\mu-{1\over 2})\log (n+\mu)\\
&\quad-{1\over 2}\left(n(n+\mu)+n-1+nz\right)(\log n +\log (n+\mu+1+z)-\log 2)\\
&\qquad+{nz\over 2}+{1\over 2}\left(d(n+\mu)+n-1\right)(\log n +\log (n+\mu+1)-\log 2)\\
&\quad-{1\over 2}\left((n+1)(n+\mu)-n-1+(n+1)z\right)(\log(n+\mu+z)-\log 2)\\
&\qquad+{(n+1)z\over 2}+{1\over 2}\left((n+1)(n+\mu)-n-1\right)(\log (n+\mu)-\log 2)\\
&\quad+{1\over 2}\left((n+1)(n+\mu)-1+(n+1)z\right)(\log (n+1)+\log(n+\mu+z)-\log 2)\\
&\qquad-{(n+1)z\over 2}-{1\over 2}\left((n+1)(n+\mu)-1\right)(\log (n+1)+\log(n+\mu)-\log 2)\\
&\quad-(n-1)(\log(n+\mu+z)-\log(n+\mu))+R_n(z),
\end{align*}
where $R_n(z)$ is given by 
\begin{align*}
R_n(z)&:=\int\limits_{0}^{\infty}{1\over t}\left({1\over 2}-{1\over t}+{1\over e^t-1}\right)\Big(e^{-(n+\mu)t}(e^{-zt}-1)-e^{-{n(n+\mu+1)t\over 2}}(e^{-{nzt\over 2}}-1)\\
&\qquad\qquad-(n+1)e^{-{(n+\mu)t\over 2}}(e^{-{zt\over 2}}-1)+e^{-(n+1)(n+\mu)t\over 2}(e^{-{z(n+1)t\over 2}}-1)\Big)\dd t.
\end{align*}
Using the inequality $|e^z-1|\leq |z|e^{|z|}$, which is valid for any $z\in \mathbb{C}$, and the fact that the function $t\mapsto {1\over t}\left({1\over 2}-{1\over t}+{1\over e^t-1}\right)$ is bounded (by $1/12$) for any $t\in\R$ it is easy to ensure that
\[
R_n(z)=O\left({|z|\over n}\right).
\]
This allows us to simplify further the expression for $T_n(z)$ and we arrive at
\begin{align*}
T_n(z)&={z\over 2}((n+1)\log(n+1)-n\log n+n(1+\log 2)-2\\
&\qquad\qquad+2\log(n+\mu+z)-n\log (n+\mu+1+z))\\
&\quad+{1\over 2}\left(n+1+2\mu\right)(\log(n+\mu+z)-\log(n+\mu))\\
&\quad-{1\over 2}\left(n(n+\mu+1)-1\right)(\log(n+\mu+1+z)-\log(n+\mu+1))+O\left({|z|\over n}\right).
\end{align*}
From a Taylor expansion of the logarithm  we deduce that
\begin{align*}
(n+1)\log(n+1)-n\log n&=\log n+1+O\left({1\over n}\right),\\
\log(k+z)-\log k&={z\over k}-{z^2\over k^2}+O\left({|z|^3\over k^3}\right),
\end{align*}
and, hence,
\begin{align*}
T_n(z)&={z\over 2}\big(n\log 2-1-\mu-(n-3)\log n\big)+O\left({1+|z|^3\over n}\right).
\end{align*}
Combining all this with \eqref{eq:MomentGenartingFct} we finally obtain
\begin{align}
\log \Ex \left[e^{zY_{\mu,n}}\right] &=z\left(\log\left( {4\Gamma\left(\frac{n}{2}\right)\over \gamma\,n!\pi^{{n-1\over 2}}}\right)+\left({\mu\over 2}+{13\over 4}\right)\log\left({n\over 2}\right)-{\mu+n+1\over 2}\right) \notag\\
&\qquad\qquad+\log \psi(z)+{z^2\over 4}\log \left({n\over 2}\right)+O\left({1+|z|^3\over n}\right) \label{eq:AsympMomentGenFct},
\end{align}
which completes the proof of Theorem \ref{thm:ModPhi} .
\end{proof}

\subsection{Large deviations principle for $Y_{\mu,n}$ for fixed $\mu$}

The purpose of this subsection is to derive large deviations principle for the sequence of random variables $Y_{\mu,n}$ when $\mu$ is fixed (recall the definition of a large deviations principle from the beginning of Section \ref{subsec:MainResults}). In particular, this covers the statement of part (iii) of Theorem \ref{thm:Intro} we presented in the introduction if we take $\mu=-1$.

\begin{theorem}\label{thm:LDP}
Let $\mu\in(-2,\infty)$ be fixed and put
\[
m_n:=-{n\over 2}\log n-{n\over 2}(\log\pi +1)+\left({\mu\over 2}+{9\over 4}\right)\log n -\log\gamma.
\] 
Then the sequence of random variables $(Y_{\mu,n}-m_n)_{n\in\N}$ satisfies a large deviations principle on $\R$ with speed ${1\over 2}\log\big({n\over 2}\big)$ and good rate function $I(x)={x^2\over 2}$.
\end{theorem}

Our proof of this result will rely on the G\"artner-Ellis theorem, see \cite[Section 2.3]{DZ}. Although this is a standard tool in large deviations theory, we reformulate it in order to keep our presentation self-contained.

\begin{lemma}[G\"artner-Ellis theorem]
Consider a sequence of random variables $(X_n)_{n\in\mathbb{N}}$ in $\R$ with logarithmic moment generating functions $\Lambda_n(t):=\log\Ex e^{tX_n}$, $t\in\R$. Let $(a_n)_{n\in\mathbb{N}}$ be a positive sequence such that $a_n\to\infty$, as $n\to\infty$. Assume that for each $t\in\R$ the limit
\[
\Lambda(t):=\lim_{n\rightarrow\infty}{1\over a_n}\Lambda_n(a_nt),
\]
exists as an extended real number. Also assume that $D_\Lambda:=\{t\in\R:\Lambda(t)<\infty\}=\R$ and that $\Lambda$ is differentiable on $D_\Lambda$. Then the sequence of random variables $X_n$ satisfies large deviations principle with speed $a_n$ and rate function $I(x)=\sup_{x\in\R}[xt-\Lambda(t)]$, the Legendre-Fenchel transform of $\Lambda$.
\end{lemma}

\begin{proof}[Proof of Theorem \ref{thm:LDP}]
Theorem \ref{thm:ModPhi} ensures that, for each $t\in\R$,
\begin{align*}
\Lambda(t) := \lim_{n\rightarrow\infty}{1\over {1\over 2}\log\big({n\over 2}\big)}\log\Ex e^{t(Y_{\mu,n}-m_n)} = {t^2\over 2}.
\end{align*}
Thus $D_\Lambda=\R$ and $\Lambda$ is differentiable on $D_\Lambda$. Moreover, the Legendre-Fenchel transform of $\Lambda$ is also given by $I(x)=x^2/2$. The G\"artner-Ellis theorem thus yields the result.
\end{proof}

\bibliographystyle{plainnat}

\end{document}